\documentclass[12pt]{article}
\usepackage[top=2cm, bottom=2cm, left=2cm, right=2cm]{geometry}
\usepackage{url}
\usepackage{subfigure}
\usepackage{amsfonts,mathrsfs}
\usepackage{amssymb,amsmath}
\usepackage{amsthm}
\usepackage{verbatim}
\usepackage{acronym}
\usepackage{graphicx}

\newcommand{\remove}[1]{}

\def\fskip#1{}

\newtheorem{theorem}{Theorem}

\newtheorem{corollary}{Corollary}

\newtheorem{definition}{Definition}
\newtheorem{example}{Example}

\newtheorem{lemma}{Lemma}

\def\1{{\bf 1}}

\def\F{\mathscr{F}}

\def\l{\lambda}

\def\g{\gamma}
\def\a{\alpha}

\def\tl{\tilde}

\def\o{\omega}

\def\R{\mathbb{R}}
\def\re{\mathbb{R}}

\def\bS{\mathbb{S}}

\newcommand{\EXP}[1]{\mathsf{E}\!\left[#1\right] }
\newcommand{\prob}{\mathsf{Pr}}

\def\argmin{\mathop{\rm argmin}}
\begin{document}
\title{On Ergodicity, Infinite Flow and Consensus in Random Models}
\author{Behrouz Touri, 
and Angelia Nedi\'c\thanks{Department of 
Industrial and Enterprise Systems Engineering, University of
Illinois, Urbana, IL 61801,
Email: \{touri1,angelia\} @illinois.edu. This research 
is supported by the National Science Foundation under 
CAREER grant CMMI 07-42538.}
}
\date{}
\maketitle
\thispagestyle{headings}
\begin{abstract}
We consider the ergodicity and consensus problem for a discrete-time
linear dynamic model driven by random stochastic matrices, which is equivalent 
to studying these concepts for the product of such matrices. 
Our focus is on the model where the random matrices
have independent but time-variant distribution. 
We introduce a new phenomenon, the infinite flow, and we study
its fundamental properties and relations with the ergodicity and consensus.
The central result is the infinite flow theorem establishing 
the equivalence between the infinite flow and the ergodicity 
for a class of independent random models, where 
the matrices in the model have a common steady state in expectation 
and a feedback property. For such models, this result demonstrates that
the expected infinite flow is both necessary and sufficient for the ergodicity.
The result is providing a deterministic characterization of
the ergodicity, which 
can be used for studying the consensus and average consensus over 
random graphs. 
\end{abstract}


\section{Introduction}\label{sec:introduction}

There is evidence of a growing number of applications in decentralized control 
of networked agents, as well as social and other networks where the consensus 
is used as a mechanism for decentralized coordination of agent actions.
The focus of this paper is on a canonical consensus problem for a linear 
discrete-time dynamic system driven by a general model of random matrices,
where the matrices are row-stochastic.  Investigating whether the model reaches
a consensus or not is often done by exploring the conditions that ensure the 
ergodicity, which in turn always guarantees the consensus.

In this paper, we propose an alternative approach by introducing a concept of
the model with infinite flow property, which can be interpreted as infinite 
information flow over time between a group of agents and the other agents in 
the network. We show that the infinite flow property is closely related to 
the ergodicity and, hence, to the consensus. In particular,
we show the equivalence between the infinite flow and the ergodicity
for a class of independent random models. 

We start by a comprehensive study of the fundamental relations and properties 
of the ergodicity, consensus and infinite flow for general random models,
independent models and independent identically distributed (i.i.d.)
random models. We then investigate the random models with a feedback property 
and the models with a common steady state for the expected matrices.
Both of these properties have been used in the analysis of consensus models,
but a deeper understanding of their roles has not been observed.
We classify feedback property in three basic types from weak to strong 
and show some relations for them.
Then, we study the models with a common steady state in expectation.
By putting all the pieces together, we show that the ergodicity of the model 
is equivalent to the infinite flow property for a class of independent random 
models with feedback property and a common steady state in expectation,
as given in infinite flow theorem (Theorem~\ref{thm:infflowthm}). 
The infinite flow theorem also establishes the equivalence 
between the infinite flow properties of the model and the expected model.
Furthermore, the theorem also shows the equivalence between the ergodicity 
of the model and the ergodicity of the expected model.
As such, the theorem provides a novel deterministic 
characterization of the ergodicity,
thus rendering another tool for studying the consensus
over random networks and convergence of random consensus algorithms.  

The main contributions of this paper include: 
1) the equivalence of the ergodicity of the model and the 
expected model for a class of independent random models with 
a feedback property and a common steady state in expectation;
2) the new insights and understanding of the ergodicity and consensus events 
over random networks brought to light through a new phenomena of infinite flow 
event, which to the best of our knowledge 
has not been known prior to this work; 
3) novel comprehensive study of the fundamental properties of
the consensus and ergodicity events for general class of 
independent random models; 
4) new insights into the role of feedback property and the role of
a common steady state in expectation for the ergodicity and consensus.

The study of the random product of stochastic matrices dates back to the early 
work in~\cite{Rosenblatt} where the convergence of the product of 
i.i.d.~random stochastic matrices 
was studied using the algebraic and topological structures of 
the set of stochastic matrices. This work was further extended 
in~\cite{Nawrotzki1,Nawrotzki2,Cogburn} 
by using results from ergodic theory of stationary processes and 
their algebraic properties.
In \cite{Tahbaz-Salehi08}, the ergodicity and consensus of the product of 
i.i.d.\ random stochastic matrices was studied using tools 
from linear algebra and probability theory, and   
a necessary and sufficient condition for 
the ergodicity was established for a class of i.i.d.\ models. 
Independently, the same problem was tackled 
in~\cite{Zampieri}, where an exponential convergence 
bound was established. 
Recently, the work in~\cite{Tahbaz-Salehi08} was extended to 
ergodic stationary processes in~\cite{Tahbaz-Salehi09}.
 
In all of the works \cite{Rosenblatt,Nawrotzki1,Nawrotzki2,Cogburn,Tahbaz-Salehi08,Zampieri,Tahbaz-Salehi09}, the underlying random models are assumed to be 
either i.i.d.\ or stationary processes, both of which imply time-invariant 
distribution on the random model. Unlike these works, 
our work in this paper is focused on the independent random models with 
{\it time-variant distributions}. Furthermore, 
we study ergodicity and 
consensus for such models using martingale and supermartingale convergence
results combined with the basic tools from probability theory. 
Our work is also related to the consensus
over random networks \cite{Hatano05}, 
optimization over random networks \cite{Lobel08}, and the consensus
over a network with random link failures \cite{Patterson09}.
Related are also gossip and broadcast-gossip schemes 
giving rise to a random consensus over a given connected bi-directional 
communication network \cite{Boyd06,Dimakis08,Aysal09,CarliFagnani10}.
On a broader basis, the paper is related to the literature on the consensus 
over networks with noisy links 
\cite{Huang07a,Huang09,KarMoura07,Touri09} 
and the deterministic consensus in decentralized systems
models \cite{Tsitsiklis84,TsitsiklisAthans84,Tsitsiklis86,Bertsekas89,
LiBasar87}, 
\cite{Jadbabaie03,Olfati-Saber04,Olshevsky09,RenBeard05} including
the effects of quantization and delay
\cite{Kashyap07,Olshevsky09a,CarliFagnani06,CarliZampieri07,Bliman_cdc08,
CarliFagnani10}. 

The paper is organized as follows. In Section \ref{sec:problem}, we describe
a discrete-time random linear dynamic system of our interest,
and introduce the ergodicity, consensus and infinite flow events.
In Section~\ref{sec:ergodicity}, we explore the relations among these events 
and establish their 0-1 law and other properties  by considering general, 
independent and i.i.d.\ random models. In Section~\ref{sec:geometry},
we discuss models with feedback properties and provide classification
of such properties with insights into their relations.
We also consider independent random model with a common steady state in 
expectation. In Section~\ref{sec:infiniteflow}, we focus on
independent random models with infinite flow property.
We establish necessary and sufficient conditions for ergodicity, and briefly 
discuss some implications of these conditions. 
We conclude in Section~\ref{sec:conclusion}.

\noindent {\bf Notation and Basic Terminology.} 
We view all vectors as columns. For a vector $x$, we write $x_i$
to denote its $i$th entry, and we write $x\ge0$ ($x>0$) to denote that 
all its entries are nonnegative (positive). We use $x^T$ to denote 
the transpose of a vector $x$. We write $\|x\|$ to  denote the standard 
Euclidean vector norm i.e., $\|x\|=\sqrt{\sum_{i}x_i^2}$.
We use $e_i$ to denote the vector with the $i$th entry equal to~1 and 
all other entries equal to~$0$, and use 
$e$ to denote the vector with all entries equal to~$1$.
A vector $a$ is  {\it stochastic} when 
$a\ge0$ and $\sum_ia_i =1$. 
We write $\{x(k)\}$ or $\{x(k)\}_{k\ge0}$ to denote a 
sequence $x(0),x(1),\ldots$ of some elements,
and we write $\{x(k)\}_{k\ge t}$ to denote the truncated sequence 
$x(t),x(t+1),\ldots$ for $t>0$. 
For a set $C$ and a subset $S$ of $C$, we write $S\subset C$ to denote
that $S$ is a proper subset of $C$. A set $S\subset C$ such that 
$S\ne\emptyset$ is referred to as a \textit{nontrivial} subset of $C$. 
We write $[m]$ to denote the integer set $\{1,\ldots,m\}$.
For a set $S\subset[m]$, we let $\bar S$ denote the 
complement set of $S$ with respect to $[m]$, i.e., 
$\bar S=\{i\in[m]\mid i\notin S\}$.

We denote the identity matrix by $I$. For a vector $v$, we 
use $diag(v)$ to denote the diagonal matrix with diagonal entries
being the components $v_i$ of the vector $v$. 
For a matrix $W$, we write  $W_{ij}$ to denote its $(i,j)$th entry, 
$W^i$ to denote its $i$th column vector, and $W^T$ to denote its transpose.
For an $m\times m$ matrix $W$, we use $\sum_{i<j}W_{ij}$ to denote
the summation of the entries $W_{ij}$ over all $i,j\in [m]$ with $i<j$.
A matrix $W$ is {\it row-stochastic} 
when its entries are nonnegative and $We=e$.
Since we deal exclusively with row-stochastic matrices, 
we will refer to such matrices simply as {\it stochastic}.
We let $\bS^m$ denote the set of $m\times m$ stochastic matrices.
A matrix $W$ is {\it doubly stochastic} when both $W$ and $W^T$ are stochastic.
 
We write $\EXP{X}$ 
to denote the expected value of a random variable $X$.
We use $\prob(\mathscr{A})$ 
and ${\it 1}_{\mathscr{A}}$ to denote the probability and
the characteristic function of an event $\mathscr{A}$, respectively.
If $\prob(\mathscr{A})=1$, we say that $\mathscr{A}$ 
happens almost surely. We often abbreviate ``almost surely'' by {\it a.s.} 

\section{Problem Formulation and Terminology}\label{sec:problem}
Throughout this article, we deal exclusively with
the matrices in the set $\bS^m$ 
of $m\times m$ stochastic matrices.
We consider the topology induced by the open sets in  
$\bS^m$ with respect to the Euclidean norm and the 
Borel sigma-algebra $\F_{\bS^m}$ of this topology. 
We assume that we are given a probability space 
$(\Omega,\mathscr{R},\prob(\cdot))$
and a measurable function $W:\Omega\to\Pi_{k=0}^\infty\,\bS^m$. 
To every $\omega\in \Omega$, the function $W(\cdot)$ is assigning a 
discrete time process $\{W(k)\}(\omega)$
in the countable product measurable space 
$\Pi_{k=0}^\infty\,(\bS^m,\F_{\bS^m})$,
where $W(k)\in\mathbb{S}^m$ is the random matrix of the process at time $k$. 
We refer to the process $\{W(k)\}(\omega)$ interchangeably 
as a \textit{random chain} or a {\it random model} and, when suitable, 
we suppress the explicit dependence on the variable $\omega$. 
We say that the chain $\{W(k)\}$ is independent if the sigma algebras generated
by the $W(k)$s for different $k\geq 0$ are independent. If  in addition 
$W(k)$s are identically distributed, then the model is 
independent identically distributed (i.i.d.).

With a given random chain $\{W(k)\}$, we associate 
a linear discrete-time dynamic system of the following form:
\begin{equation}\label{eqn:dynsys}
x(k+1)=W(k)x(k)\qquad\hbox{for $k=0,1,2,\ldots$},
\end{equation}
where $x(k)\in\mathbb{R}^m$ is a state vector at time $k$ and  
$x(0)$ is the initial state vector.
We will often refer to the system in (\ref{eqn:dynsys}) as 
{\it the dynamic system driven by the chain $\{W(k)\}$}.

We are interested in providing conditions guaranteeing that 
the dynamic system reaches a consensus almost surely. 
Since reaching the consensus is closely related to the 
ergodicity of the chain, we are also interested in studying the ergodicity
on the fundamental level. In our study of the random consensus and ergodicity, 
we use another property of the chain, an infinite flow property. 
We start by providing these basic notions for 
a deterministic chain. 

\begin{definition}\label{def:basic}
Given a deterministic chain $\{G(k)\}\subset\mathbb{S}^m$, we say that:\\
The system $z(k+1)=G(k)z(k)$ {\it reaches a consensus}
if for any initial state $z(0)\in\re^m$, 
there exists a scalar $c(z(0))$ such that 
$\lim_{k\to\infty} \|z(k)-c(z(0))\,e\|=0.$\\
The chain $\{G(k)\}$ is 
\textit{ergodic} if for any $t\ge0$ and $j\in[m]$,
there is a scalar $g_j(t)$ such that 
\[\lim_{k\rightarrow \infty}\Phi_{ij}(k,t)=g_j(t)\qquad
\hbox{for all $i\in [m]$},\]
where $\Phi(k,t)=G(k)G(k-1)\cdots G(t)$ for $k>t$ and $t\ge0$.\\
The chain $\{G(k)\}$
{\it has infinite flow} property if 
$\sum_{k=0}^\infty \sum_{i\in S,\, j\in \bar S} 
\left(G_{ij}(k) + G_{ji}(k)\right)=\infty$
for any nontrivial subset $S\subset[m]$.
\end{definition}

In the definition of the infinite flow property,  
the quantity $\sum_{i\in S,\, j\in \bar S} 
\left(G_{ij}(k) + G_{ji}(k)\right)$ can be interpreted as 
a flow between the subset $S$ and its complement $\bar S$ 
in a weighted graph. 
In particular, consider the undirected weighted graph 
$\mathcal{G}(k)$ with the node set $[m]$, the edge set induced 
by the positive entries in $G(k)+G^T(k)$, 
and the weight matrix $G(k)+G^T(k)$.
Then, the quantity $\sum_{i\in S,\, j\in \bar S} 
\left(G_{ij}(k) + G_{ji}(k)\right)$ represents the
flow in graph $\mathcal{G}(k)$ across the cut $(S,\bar S)$ 
for a nontrivial node set 
$S\subset[m]$ and its complement $\bar S$. 
For the graphs $\mathcal{G}(k)$ induced by the matrices $G(k)$,
the infinite flow property requires that
the total flow in time across any nontrivial cut $(S,\bar S)$
is infinite, which could be viewed as infinite information exchange
between the nodes in $S$ and $\bar S$.

The ergodicity is equivalent to 
the following condition~\cite{SenetaCons}: for any $t\ge0$ and $x\in\re^m$, 
there is a scalar $\gamma(t,x)$ such that
$\lim_{k\rightarrow \infty}\Phi(k,t)x=\gamma(t,x)\, e.$
Since the matrices $G(k)$ have finite dimension,
the ergodicity is also equivalent to the following condition:  
for any $t\ge0$ and any $\ell\in[m]$, there is a scalar $\gamma_\ell(t)$ 
such that $\lim_{k\rightarrow \infty}\Phi(k,t)e_\ell=\gamma_\ell(t)\, e.$
Also, due to the linearity and finite dimension
of the system $z(k+1)=G(k)z(k)$,
the consensus can be studied by considering
only the initial states $x(0)=e_\ell$, $\ell\in[m]$, rather than
all $x(0)\in\re^m$. 

Clearly, the ergodicity of the chain implies reaching a consensus.
However, a consensus may be reached even if
the chain $\{G(k)\}$ is not ergodic, as seen in the following example.

\begin{example}\label{exam:noerg}
Let $G(0)=ev^T$ for a stochastic vector $v$, and 
let $G(k)=I$ for all $k\ge1$.
Then, we have $\Phi(k,0)=ev^T$ for all $k\ge 1,$ implying that 
the system $x(k+1)=G(k)x(k)$ reaches a consensus.
However, the chain $\{G(k)\}$ is not ergodic since $\Phi(k,t)=I$ 
for any $k>t\ge1$. $\square$
\end{example}

Using Definition~\ref{def:basic}, we now introduce the 
corresponding events of consensus, ergodicity and infinite flow.
Given a random chain $\{W(k)\}$, let $\mathscr{C}$  
denote the event that the system in~(\ref{eqn:dynsys}) 
reaches a consensus for any initial state $x(0)$.
Let $\mathscr{E}$ denote the event that the chain $\{W(k)\}$ is ergodic, and
let $\mathscr{F}$ denote the event 
that the chain has the {\it infinite flow property}.
We refer to $\mathscr{C}$, $\mathscr{E}$ and $\mathscr{F}$
as the {\it consensus event}, the {\it ergodicity event} and 
{\it the infinite flow event}, respectively.
We say that 
the {\it model is ergodic} if the ergodicity event $\mathscr{E}$ occurs 
almost surely. The {\it model admits consensus}  
if the consensus event $\mathscr{C}$ occurs almost surely.
The {\it model has infinite flow}
if the infinite flow event $\mathscr{F}$ occurs almost surely.
The {\it model has expected infinite flow} if its expected chain
$\{\EXP{W(k)}\}$ has infinite flow. 

\section{Infinite Flow, Ergodicity  and Consensus}\label{sec:ergodicity}
In this section, we further study the ergodicity event $\mathscr{E}$,
the consensus event $\mathscr{C}$ and the infinite flow event $\mathscr{F}$
under different assumptions on the nature of the randomness in the model.
In particular, in Section~\ref{sec:relations} we establish some fundamental 
relations among $\mathscr{E}$, $\mathscr{C}$ and $\mathscr{F}$. 
In Section~\ref{sec:0-1laws}, we investigate the 0-1 law properties of 
these events, while in Section~\ref{sec:expected} we provide some relations for
a random model and its expected model.

Throughout the rest of the paper, we use the following notation.
For a stochastic matrix $W$ and a nontrivial subset $S\subset[m]$, 
we define $W_S$ as follows: 
\begin{equation}\label{eqn:flowWS}
W_S=\sum_{i\in S,j\in \bar{S}} \left(W_{ij}+ W_{ji}\right).
\end{equation}
Note that for a given a random model $\{W(k)\}$,
the infinite flow event $\mathscr{F}$ is given by
\begin{equation}\label{eqn:infflow} 
\mathscr{F}=\bigcap_{S\subset[m]}
\left\{\sum_{k=0}^{\infty}W_S(k)=\infty\right\}.
\end{equation}
Note also that the infinite flow event requires that, across any nontrivial 
cut $(S,\bar S)$, the total flow in the random graphs induced by the matrices 
$W(k)+W^T(k), k\ge0,$ is infinite.

\subsection{Basic Relations}\label{sec:relations}
As discussed in Section~\ref{sec:problem}, 
we have $\mathscr{E}\subseteq \mathscr{C}$ for any random model.
We here show that the ergodicity event is also
always contained in the infinite flow event, i.e., 
$\mathscr{E}\subseteq \mathscr{F}$. 
We establish this by using the following result for a deterministic model.

\begin{lemma}\label{lemma:flowmaxminS}
Let $\{A(k)\}\subset\bS^m$ be a deterministic sequence, and 
let $\{z(k)\}$ be generated by $z(k+1)=A(k)z(k)$ for all $k\ge0$
with an initial state $z(0)\in\mathbb{R}^m$.
Then, for any nontrivial subset $S\subset[m]$ and $k\ge0$, we have
\[\max_{i\in S} z_i(k+1)\le \max_{s\in S}z_s(0)+d(z(0))\,\sum_{t=0}^kA_S(t),\]
\[\min_{j\in\bar S} z_j(k+1)\ge \min_{r\in \bar{S}}z_r(0)-d(z(0))\,
\sum_{t=0}^k A_S(t),\]
where $d(y)=\max_{\ell\in[m]} y_\ell-\min_{r\in[m]} y_r$ for $y\in\re^m.$
\end{lemma}

\begin{proof}
Let $S\subset[m]$ be an arbitrary nontrivial set and let $k\ge0$ be arbitrary. 
Let $z_{\min}(k)=\min_{r\in[m]} z_r$ and $z_{\max}(k)=\max_{s\in[m]} z_s(k)$.
Since $z_i(k+1)=\sum_{\ell=1}^m A_{i\ell}(k) z_\ell(k)$,
by the stochasticity of $A(k)$ we have 
$z_i(k)\in[z_{\min}(0),z_{\max}(0)]$ for all $i\in[m]$ and all $k$.  
Then, we obtain for $i\in S$,
\begin{align}\nonumber
&z_i(k+1)=\sum_{\ell\in S} A_{i\ell}(k)z_\ell(k)
+\sum_{\ell\in \bar{S}}A_{i\ell}(k)z_\ell(k)\cr 
&\quad\le \sum_{\ell\in S}A_{i\ell}(k) \left(\max_{s\in S}z_{s}(k)\right)
+z_{\max}(0)\sum_{\ell\in \bar{S}}A_{i\ell}(k),
\end{align}
where the inequality follows by $A_{i\ell}(k)\ge0$. 
By the stochasticity of $A(k)$, we also obtain
\begin{align*}
&z_i(k+1)
\leq \left(1-\sum_{\ell\in \bar S} A_{i\ell}(k)\right)
\max_{s\in S}z_s(k) \cr 
&\qquad\qquad+z_{\max}(0)\sum_{\ell\in \bar{S}}A_{i\ell}(k)\cr 
&=\max_{s\in S}z_s(k)+\left(z_{\max}(0)-\max_{s\in S}z_s(k)\right) 
\sum_{\ell\in \bar S}A_{i\ell}(k).\end{align*}
By the definition of $A_S$ in~\eqref{eqn:flowWS}, we have 
$0\le \sum_{\ell\in \bar S}A_{i\ell}(k)\leq A_S(k)$.
Since $z_{\max}(0)-\max_{s\in S}z_s(k)\ge0$, it follows 
\begin{align*}
z_i(k+1)&
\le \max_{s\in S}z_s(k)+(z_{\max}(0)-\max_{s\in S}z_s(k))A_S(k)
\cr 
&\le\max_{s\in S}z_s(k)+d(z(0))A_S(k),\end{align*}
where the last inequality holds since 
$z_{\max}(0)-\max_{s\in S}z_s(k)\leq z_{\max}(0)-z_{\min}(0)=d(z(0))$. 
By taking the maximum over all $i\in S$ in the preceding relation
and by recursively using the resulting inequality,
we obtain $\max_{i\in S}z_i(k+1)\le \max_{s\in S}z_s(k)+d(z(0))A_S(k)$ 
and recursively, we get $\max_{i\in S}z_i(k+1)\leq \max_{s\in S}z_s(0) 
+d(z(0))\sum_{t=0}^k A_S(t)$.

The relation for $\min_{j\in\bar S} z(k+1)$ follows from the preceding 
relation 
by considering $\{z(k)\}$ generated with the starting point $-z(0)$.
\end{proof}

Using Lemma~\ref{lemma:flowmaxminS}, we now show that 
the ergodicity event is contained in the infinite flow event.

\begin{theorem}\label{thm:infiniteflow}
Let $\{A(k)\}\subset\mathbb{S}^m$ be an ergodic  deterministic chain. 
Then $\sum_{k=0}^{\infty}A_S(k)=\infty$ for any nontrivial $S\subset[m]$.
In particular, we have $\mathscr{E}\subseteq \mathscr{F}$ 
for any random model. 
\end{theorem}
\begin{proof}
To arrive at a contradiction, assume that there is a nontrivial set 
$S\subset[m]$ such that $\sum_{k=0}^{\infty}A_S(k)<\infty$. 
Since the matrices $A(k)$ are stochastic, we have $A_S(k)\ge0$ for all $k$. 
Therefore, there exists large enough $\bar t\geq 0$ such that 
$\sum_{k=\bar t}^{\infty}A_S(k)<\frac{1}{4}$.

Now, define the vector $\bar z=(\bar z_1,\ldots,\bar z_m)^T$, 
where $\bar{z}_i=0$ for $i\in S$ and $\bar{z}_i=1$ for $i\in\bar{S}$. 
Consider the dynamic system $z(k+1)=A(k)z(k)$ for $k\ge\bar t$, 
which is started at time $\bar t$ in state $z(\bar t)=\bar z$.
Note that Lemma~\ref{lemma:flowmaxminS} applies to the case 
where the time $t=\bar t$ is
taken as initial time, in which case $d(0)$ corresponds to 
$d(\bar t)=\max_i z_i(\bar t)-\min_j z_j(\bar t)$.
Also, note that $d(\bar t)=1$
by the definition of the starting state $\bar z$. Thus, by
applying Lemma~\ref{lemma:flowmaxminS}, 
we have for all $k\ge\bar t$, 
$\max_{i\in S} z_i(k+1)\le \max_{s\in S}z_s(\bar t)+
\sum_{t=\bar t}^kA_S(t)$ and 
$\min_{j\in\bar S} z_j (k+1)\ge \min_{r\in \bar{S}}z_r(\bar t)-
\sum_{t=\bar t}^k A_S(t)$. 
Since $\max_{s\in S}z_s(\bar t)=0$ and $\min_{r\in \bar{S}}z_r(\bar t)=1$,
it follows that 
$\max_{i\in S} z_i(k+1)\le \sum_{t=\bar t}^kA_S(t)$ and 
$\min_{j\in\bar S} z_j (k+1)\ge 1-
\sum_{t=\bar t}^k A_S(t)$. Using these relations and 
$\sum_{k=\bar t}^{\infty}A_S(k)<\frac{1}{4}$,
we have $\liminf_{k\rightarrow\infty} \left(z_j(k)-z_i(k)\right)> 1-
2\sum_{t=\bar t}^\infty A_S(t)=\frac{1}{2}$ 
for any $j\in\bar S$ and $i\in S$,  
thus showing that the chain $\{A(k)\}$ is not ergodic - a contradiction. 
Therefore, we must have $\sum_{t=0}^\infty A_S(t)=\infty$ for any 
nontrivial $S\subset[m]$.

From the preceding argument and the definitions of  
$\mathscr{E}$ and $\mathscr{F}$, we conclude that $\o\in \mathscr{E}$ 
implies $\o\in\mathscr{F}$ for any random model $\{W(k)\}$.
Hence, $\mathscr{E}\subseteq \mathscr{F}$ for any random model.
\end{proof}

Theorem~\ref{thm:infiniteflow} shows that an ergodic model must have an 
infinite flow property. In other words, the infinite flow property of 
any random model is necessary for the ergodicity of the model. 
Later in Theorem~\ref{thm:sufficient}, for a certain class of random models,
we will show that the infinite flow is also sufficient for the ergodicity.

Figure~\ref{fig:events_indep_model} illustrates the inclusions 
$\mathscr{E}\subseteq\mathscr{C}$ and $\mathscr{E}\subseteq\mathscr{F}$
for a general random model. The inclusion 
$\mathscr{E}\subseteq \mathscr{C}\cap \mathscr{F}$ in 
Figure~\ref{fig:events_indep_model} can be strict as seen
in the following example.

\begin{example}
Consider the $2\times 2$ chain $\{A(k)\}$ defined by 
\[A(0)=\left[\begin{array}{cc}\frac{1}{2}&\frac{1}{2}\\\frac{1}{2}&\frac{1}{2}
\end{array}\right], \ \hbox{ and } \ 
A(k)=\left[\begin{array}{cc}0&1\\1&0
\end{array}\right]\ \hbox{for $k\geq 1$}.\]
For any stochastic matrix $B$, we have $BA(0)=A(0)$ and hence, 
the model admits consensus. Furthermore, 
the model has infinite flow property. However, the chain 
$\{A(k)\}_{k\geq 1}$ does not admit consensus. 
Therefore, in this case $\mathscr{C}\cap \mathscr{F}=\Omega$ 
(the entire space of realizations) while $\mathscr{E}=\emptyset$. 
$\square$
\end{example}

\begin{figure}
\vskip -1pc
 \begin{center}
   \includegraphics[width=0.4\linewidth]{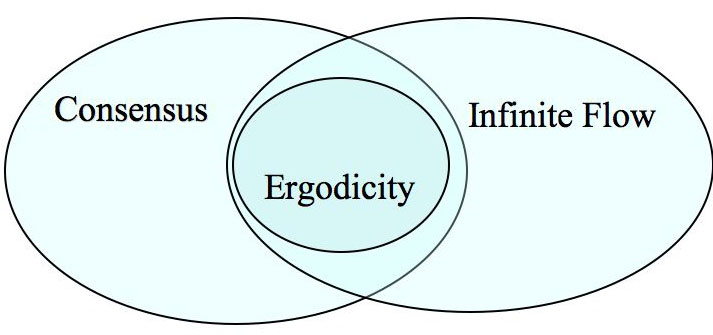}
 \end{center}\vspace{-0.4cm}
\caption{\small The relations among the ergodicity, consensus,
and infinite flow events for a general random model.}
\label{fig:events_indep_model}\vspace{-0.7cm}
\end{figure}
Next, we provide a sufficient condition
for the ergodicity and consensus events to coincide. 

\begin{lemma}\label{lemma:fulrank}
Let $\{W(k)\}$ be a (not necessarily independent) random chain such that 
each matrix $W(k)$ is invertible almost surely. 
Then, we have $\mathscr{E}=\mathscr{C}$ almost surely.
\end{lemma}

\begin{proof}
The inclusion $\mathscr{E}\subseteq \mathscr{C}$ follows from the definition,
so it suffices to show $\mathscr{C}\subseteq\mathscr{E}$ almost surely.
In turn, to show $\mathscr{C}\subseteq\mathscr{E}$ almost surely, it 
suffices to prove that $\mathscr{C}\cap R\subseteq\mathscr{E}$ for a set $R$
such that $\prob(R)=1$.
For each $k\ge0$, let $R_k$ be the set of instances $\omega$ 
such that the matrix $W(k)(\omega)$ is invertible. 
Define $R=\cap_{k=0}^{\infty}R_k$. We have $\prob(R_k)=1$ for each $k\ge0$ 
by our assumption that each matrix $W(k)$ is invertible almost surely.
In view of this and the fact that 
the collection $\{R_k\}$ is countable, it follows that $\prob(R)=1$. 

We next show that $\mathscr{C}\cap R\subseteq\mathscr{E}$.
Let $\o\in \mathscr{C}\cap R$ so that $W(k)(\o)$ has full rank for all 
$k\geq 0$. To simplify notation, let $\tl W(k)= W(k)(\o)$.
Consider an arbitrary starting time $s\geq 0$. 
We show that the consensus is reached for the dynamic 
$z(t) = \tl W(t-1)z(t-1)$ with $t>s$, i.e., 
for any $z(s)\in\mathbb{R}^m$, we have 
$\lim_{t\rightarrow\infty}z(t) = c e$ for some $c\in \mathbb{R}$.
For a given $z(s)\in \mathbb{R}^m$, define $\tl x(0) =
[\tl W(s-1)\cdots \tl W(1) \tl W(0)]^{-1} z(s)$ and 
consider the dynamic $\tl x(k) =\tl W(k-1)\tl x(k-1)$ 
started at time $t_0=0$ with the initial vector $\tl x(0)$. 
By the definition of $\tl x(0)$, we have
$\tl x(s)=\tl W(s-1)\cdots \tl W(1) \tl W(0)\tl x(0)=z(s)$. Therefore, 
for all $t> s$,  
\begin{align*}\tl x(t)&= \tl W(t-1)\cdots \tl W(s-1)\cdots \tl W(1)\tl W(0)\tl x(0)\cr 
&=\tl W(t-1)\cdots \tl W(s)z(s)=z(t).
\end{align*}
By the definition of $\tl W(k)$, we have 
$\lim_{t\rightarrow\infty}\tl x(t) = c e$ 
for some $c\in\re$ (since $\o\in \mathscr{C}$). Therefore, it follows that 
$\lim_{t\rightarrow\infty}z(t) = c e$, thus showing that
the dynamic system $z(t)= \tl W(t-1)z(t-1)$, $t> s,$ 
reaches a consensus. Since this is true for arbitrary $s\ge0$ and  
$z(s)\in\mathbb{R}^m$, the chain $\{W(k)\}(\o)$ is ergodic, which implies 
$\o\in\mathscr{E}$.
\end{proof}

In general, there may be no further refinements of inclusion 
relations among the events $\mathscr{E},\mathscr{F}$
and $\mathscr{C}$ even when the model is independent, as indicated by
the following example. 
\begin{example}\label{ex:cons_nontriv}
Consider an independent random model where for $p\in(0,1]$, we have 
$W(0)=\frac{1}{m}ee^T$ with probability $p$, $W(0)=I$ with probability $1-p$s 
and $W(k)=I$ with probability~$1$ for all $k\ge 1.$
In this case, the consensus event $\mathscr{C}$ 
happens with probability~$p>0$.
However, the infinite flow and the ergodicity events are empty sets. $\square$
\end{example}

Example~\ref{ex:cons_nontriv} shows that we can have 
$\mathscr{C}\cap\mathscr{E}=\emptyset$
and $\mathscr{C}\cap\mathscr{F}=\emptyset$, while $\mathscr{C}\ne\emptyset$. 
Thus, even for an independent model the consensus event need not be
contained in either $\mathscr{E}$ or $\mathscr{F}$.
However, if we further restrict our attention to i.i.d.~models, 
we can show that $\mathscr{E}=\mathscr{C}$ almost surely. To establish this,
we make use of the following lemma.

\begin{lemma}\label{lemma:uniformdec}
Let $A\in\bS^m$ and $x\in\re^m$. Also, let $A$ be such that 
$\max_{i\in [m]} [Ae_\ell]_i-\min_{j\in[m]}[Ae_\ell]_j\leq \frac{1}{2m}$ 
for any $\ell\in[m]$, where $[v]_i$ denotes the $i$th 
component of a vector $v$.
Then, we have $\max_{i}[Ax]_i-\min_{j}[Ax]_j\leq \frac{1}{2}$ 
for any $x\in[0,1]^m$.
\end{lemma}
\begin{proof}
Let $x\in\re^m$ with $x_\ell\in[0,1]$ for any $\ell\in[m]$. 
Then, we have for any $i,j\in[m]$, 
\begin{align*}
y_i-y_j&=\sum_{\ell=1}^m(A_{i\ell}-A_{j\ell})x_\ell
\leq \sum_{\ell=1}^m|A_{i\ell}-A_{j\ell}|\cr 
&=\sum_{\ell=1}^m \left| [Ae_\ell]_i-[Ae_\ell]_j \right|.
\end{align*}
By the assumption on $A$, we obtain
$\left| [Ae_\ell]_i-[Ae_\ell]_j \right|\le 
\max_{i\in [m]}[Ae_\ell]_i-\min_{j\in[m]}[Ae_\ell]_j\le \frac{1}{2m}$. Hence, 
$y_i-y_j\leq\sum_{\ell=1}^m\frac{1}{2m}=\frac{1}2$, implying 
$\max_{i}y_i-\min_{j}y_j\leq \frac{1}{2}$.
\end{proof}

We now provide our main result for i.i.d.~models, which states that
the ergodicity and the consensus events are almost surely equal.
We establish this result by using Lemma~\ref{lemma:uniformdec} and 
the Borel-Cantelli lemma (see \cite{Durrett}, page 50).

\begin{theorem}\label{thm:consensus01}
We have $\mathscr{E}=\mathscr{C}$
almost surely for any i.i.d.\ random model. 
\end{theorem}
\begin{proof}
Since $\mathscr{E}\subseteq\mathscr{C}$, the assertion is true
when consensus occurs with probability $0$. Therefore, it suffices to show that
if the consensus occurs with a probability $p$ other than 0, 
the two events are almost surely equal.
Let $\prob(\mathscr{C})=p$ with $p\in(0,1]$. Then, 
for all $\o\in\mathscr{C}$,  
\[\lim_{k\to\infty} d(x(k))(\o)=0\quad\hbox{where \ }
d(x)=\max_i x_i-\min_j x_j,\]
and $\{x(k)\}(\omega)$ is the sequence
generated by the dynamic system~\eqref{eqn:dynsys} with any 
$x(0)\in\mathbb{R}^m$. 

For every $\ell\in[m]$, let $\{x^\ell(k)\}$ be the sequence generated by 
the dynamic system in~\eqref{eqn:dynsys} with $x(0)=e_\ell$.
Then, for any $\o\in\mathscr{C}$,
there is the smallest integer $K^\ell(\o)\ge0$ such that 
\[d(x^\ell(k))(\o)\le \frac{1}{2m}
\qquad\hbox{for all }k\ge K^\ell(\o).\]
Note that $d(x^\ell(k))(\o)$
is a nonincreasing sequence (of $k$) for each $\ell\in[m]$. Hence, by letting 
$K(\o)=\max_{\ell\in[m]}K^\ell(\o)$ we obtain 
$d(x^\ell(k))(\o)\le \frac{1}{2m}$
for all $\ell\in[m]$ and $k\ge K(\o)$. Thus, by applying 
Lemma~\ref{lemma:uniformdec}, we have for 
almost all $\o\in\mathscr{C}$, 
\begin{equation}\label{eqn:dkd0}
d(x(k))(\o)\le \frac{1}{2},
\end{equation}
for all $k\ge K(\o)$ and 
$x(0)\in [0,1]^m$. By the definition of consensus, we have
$\lim_{N\to\infty} \prob(K\le N)\ge\prob(\mathscr{C})=p$.
Thus, by the continuity of the measure, there exists an integer $N_1$
such that $\prob(K< N_1)\ge \frac{p}{2}.$

Now, let time $T\geq 0$ be arbitrary, and 
let $l_k^T$ denote the $N_1$-tuple 
of the matrices $W(s)$ driving the system~\eqref{eqn:dynsys}
for $s=T+N_1k,\ldots, T+N_1(k+1)-1$ and $k\ge0$,
i.e.,
\begin{align*}
l^T_k&=\Bigl( W(T+N_1 k), W(T+N_1k +1),\cr 
&\qquad \ldots,W(T+N_1(k+1)-1)\Bigr)
\quad \hbox{for all $k\ge0$}.
\end{align*}
Let $L_{N}$ denote the collection of all $N$-tuples
$(A_1,\ldots,A_N)$ of matrices $A_i\in\mathbb{S}^m$, $i\in[N]$ 
such that for $x(N)=A_NA_{N-1}\cdots A_1x(0)$ with $x(0)\in[0,1]^m$, 
we have $d(x(N))\le\frac{1}{2}.$
By the definitions of $l^T_k$ and $L_N$, 
relation~\eqref{eqn:dkd0} and relation
$\prob(K< N_1)\ge \frac{p}{2}$ state that
$\prob(\{l^T_0\in L_{N_1}\})\ge\frac{p}{2}.$ 
By the i.i.d.~property of the model,
the events $\{l^T_k\in L_{N_1}\}$, $k\ge0$, are i.i.d.\ and the probability of 
their occurrence is equal to $\prob(\{l^T_0\in L_{N_1}\})$, implying that 
$\prob(\{l^T_k\in L_{N_1}\})\ge\frac{p}{2}$ for all $k\ge0$. 
Consequently, $\sum_{k=0}^\infty \prob(\{l^T_k\in L_{N_1}\})=\infty.$
Since the events $\{l^T_k\in L_{N_1}\}$ are~i.i.d.,
by Borel-Cantelli lemma 
$\prob(\{\o\in\Omega\mid \o\in \{l^T_k\in L_{N_1}\}\  i.o.\})=1,$
where $i.o.$~stands for infinitely often. 
Observing that the event 
$\{\o\in\Omega\mid \o\in \{l^T_k\in L_{N_1}\}\  i.o.\}$ is contained in
the consensus event for the chain $\{W(T+k)\}_{k\geq 0}$,
we see that the consensus event for the chain $\{W(T+k)\}_{k\geq 0}$
occurs almost surely. Since this is true for arbitrary $T\ge0$ 
it follows that the chain $\{W(k)\}$ is ergodic {\it a.s.},
implying $\mathscr{C}\subseteq\mathscr{E}$ {\it a.s.}
This and the inclusion $\mathscr{E}\subseteq\mathscr{C}$ yield
$\mathscr{C}=\mathscr{E}$ {\it a.s.} 
\end{proof}

Theorem~\ref{thm:consensus01} extends the 
equivalence result between the consensus and ergodicity for i.i.d.~models 
given in Theorem~3.a and Theorem~3.b of \cite{Tahbaz-Salehi08} 
(and hence Corollary~4 in \cite{Tahbaz-Salehi08}), which are established there
assuming that the  matrices have positive diagonal entries almost surely. 
The relations among $\mathscr{C}$, $\mathscr{E}$, and $\mathscr{F}$
for i.i.d.\ case are illustrated in Figure~\ref{fig:events_01law}.

\subsection{0-1 Laws}\label{sec:0-1laws}
In this section, we discuss 0-1 laws for the events 
$\mathscr{E}$, $\mathscr{F}$ and $\mathscr{C}$ for independent random models.
The 0-1 laws specify the trivial (or 0-1) events, which are 
the events occurring with  either probability~0 or~1. 
The ergodicity event is a 0-1 event, 
as shown\footnote{Even though 
the result there was stated assuming a more restrictive random model, 
the  proof itself relies only on the independence property of the model.}
in~\cite{Tahbaz-Salehi08}, Lemma~1. 
Since the ergodicity event is always contained in the
consensus event, the ergodicity event $\mathscr{E}$ occurs with probability 0
whenever the consensus event $\mathscr{C}$ occurs with a probability 
$p\in(0,1)$. In other words, we may have $\prob(\mathscr{E})=1$ 
only if $\prob(\mathscr{C})=1$.

We next show that the infinite flow is also a 0-1 event. 
\begin{lemma}\label{lemma:flow_trivial}
For an independent random model, the infinite flow event $\mathscr{F}$ 
is a 0-1 event. 
\end{lemma}
\begin{proof}
For a nontrivial $S\subset[m]$, the sequence $\{W_S(k)\}$ of 
undirected flows across the cut $(S,\bar S)$ (see Eq.~(\ref{eqn:flowWS})) 
is a sequence of independent (finitely valued) random variables. 
The event $\left\{\sum_{k=0}^{\infty}W_S(k)=\infty\right\}$ is a tale 
event and, by Kolmogorov's 0-1 law (\cite{Durrett}, page 61), 
this event is a 0-1 event. 
Since there are finitely many nontrivial sets $S\subset [m]$, the event 
$\mathscr{F}=\bigcap_{S\subset[m]}\left\{\sum_{k=0}^{\infty}W_S(k)
=\infty\right\}$ is also a 0-1 event. 
\end{proof}

While both events $\mathscr{E}$ and $\mathscr{F}$ are trivial for an 
independent model, 
the situation is not the same for the consensus event~$\mathscr{C}$. 
In particular, by Example~\ref{ex:cons_nontriv} where $p\in(0,1)$, we see that 
the consensus event need not assume 0-1 law since it can occur with a 
probability $p\in(0,1)$. 

However, the situation is very different for i.i.d.\ models.
In particular,  in this case the consensus event is also a trivial event, 
as seen in the following lemma.

\begin{lemma}
For an i.i.d.\ random model, the consensus event $\mathscr{C}$ is a 0-1 event. 
\end{lemma}
\begin{proof}
The result follows from the fact that  $\mathscr{E}$ is a trivial
event and Theorem~\ref{thm:consensus01}, which states that 
$\mathscr{E}=\mathscr{C}$ almost surely for i.i.d.\ models. 
\end{proof}

Figure~\ref{fig:events_01law} illustrates the 0-1 laws of
$\mathscr{E},$ $\mathscr{C}$ 
and $\mathscr{F}$ for an i.i.d.\ model.
Our next example demonstrates that the inclusion relation in 
Figure~\ref{fig:events_01law} can be strict. 

\begin{figure}
 \begin{center}
  \includegraphics[width=0.4\columnwidth]{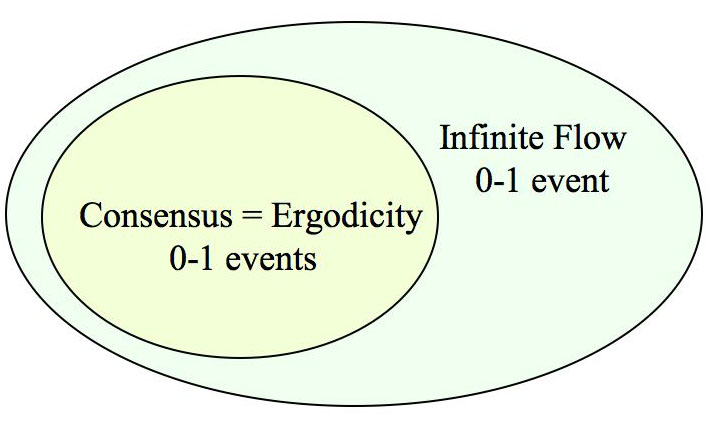}
 \end{center}\vspace{-0.4cm}
\caption{\small The ergodicity and consensus coincide {\it a.s.},
and all three events assume 0-1 law for an i.i.d\ model.}
\label{fig:events_01law}\vspace{-0.7cm}
\end{figure}

\begin{example}\label{ex:cons_empty}
Consider the independent identical random model where 
each $W(k)$ is equally likely to be any of the $m\times m$ 
permutation matrices. Then, in view of the uniform distribution, 
we have $\EXP{W(k)}=\frac{1}{m}\,ee^T$ for all $k$. 
Hence, by Theorem~\ref{thm:expectedflow}, 
it follows that the infinite flow event $\mathscr{F}$ 
is happening almost surely. 
But, since the chain $\{W(k)\}$ is a sequence of permutation matrices, 
the consensus event $\mathscr{C}$ never happens. $\square$
\end{example}

\subsection{Random Model and Its Expected Model}\label{sec:expected}
Here, we investigate the properties of an independent random model and 
its corresponding expected model. We establish two results
in forthcoming Theorems~\ref{thm:expectedflow} and~\ref{thm:equivalency}
that later on play an important role in the establishment the Infinite
Flow Theorem in Section~\ref{sec:infiniteflow}.
The first result shows the equivalence
of the infinite flow property for a random chain $\{W(k)\}$ and its
expected chain $\{\EXP{W(k)}\}$, as given in the following theorem.
 
\begin{theorem}\label{thm:expectedflow}
Let $\{W(k)\}$ be an independent random model.
Then, the model has infinite flow property
if and only if the expected model has infinite flow  property.
\end{theorem}
\begin{proof}
Let $S\subset[m]$ be nontrivial. 
Since the model is independent, the random variables $W_S(k)$ are independent.
By the definition of $W_S(k)$ and the stochasticity of $W(k)$,
we have $0\leq W_S(k)\leq \sum_{i=1}^m\sum_{j=1}^mW_{ij}(k)=m$ for all $k\ge0$.
Thus, by monotone convergence theorem (\cite{Durrett}, page 225), 
the infinite flow property of the model implies the infinite flow property of 
the expected model. On the other hand,
since the model is independent and $0\leq W_S(k)\leq m$, 
by Kolmogorov's three-series theorem (\cite{Durrett}, page 64) it follows that:
if $\sum_{k=0}^\infty \EXP{W_S(k)}=\infty$, then 
$\prob\left(\sum_{k=0}^{\infty}W_S(k)=\infty\right)>0$.
Since $\sum_{k=0}^{\infty}W_S(k)$ is a trivial event, we have
$\prob\left(\sum_{k=0}^{\infty}W_S(k)=\infty\right)=1$. 
Thus, since $S\subset[m]$ is arbitrary, 
the model has infinite flow property. 
\end{proof}

There is no analog of Theorem~\ref{thm:expectedflow} for the ergodicity 
or consensus event, unless additional assumptions are imposed.
However, a weaker result holds as seen in the following. 

\begin{lemma}\label{lemma:experg}
Let $\{W(k)\}$ be an independent model and assume
that the model admits consensus (is ergodic).
Then, the expected model $\{\EXP{W(k)}\}$ reaches a consensus (is ergodic). 
\end{lemma}
The proof\footnote{Assuming more restrictive assumptions on the model, 
the result of the lemma was stated in~\cite{Tahbaz-Salehi08} (Theorem~3, 
$(a)\Rightarrow(b)$) and~\cite{Zampieri} (Remark~3.3). 
However, the proofs there rely only on the 
independence of the model.} of Lemma~\ref{lemma:experg} can be found in 
\cite{Tahbaz-Salehi08,Zampieri}.

The following theorem states another important result for later use. 
As a consequence of Lemma~\ref{lemma:experg}, and 
Theorems~\ref{thm:infiniteflow} and~\ref{thm:expectedflow},
the result provides an equivalent deterministic 
characterization of the ergodicity for a class of independent 
models. 

\begin{theorem}\label{thm:equivalency}
Let $\{W(k)\}$ be an independent random model 
such that $\mathscr{E}=\mathscr{F}$ almost surely. 
Then, the model is ergodic 
if and only if the expected model is ergodic.
\end{theorem}
\begin{proof}
If the ergodicity event is almost sure,
then by Lemma~\ref{lemma:experg}, the expected model is ergodic.  
For the converse statement, let the chain $\{\EXP{W(k)}\}$ be
ergodic. Then, by Theorem~\ref{thm:infiniteflow} 
the chain $\{\EXP{W(k)}\}$ has infinite flow. Therefore, 
by Theorem~\ref{thm:expectedflow} the infinite flow event $\mathscr{F}$ 
is almost sure, and since $\mathscr{E}=\mathscr{F}$ {\it a.s.}, 
the ergodicity event is almost sure.
\end{proof}

\section{Model with Feedback Property and Steady State in Expectation} 
\label{sec:geometry}
In this section, we discuss two  properties of a random model 
that are important in the development of our main results in 
Section~\ref{sec:infiniteflow}. In particular, we introduce and 
study a model with feedback properties and a model with a common steady state 
in expectation. Stronger forms of these properties have always been used 
when establishing consensus both for deterministic and random models.
Here, we provide some new fundamental insights into these properties. 

\subsection{Feedback Properties}\label{sec:feedback}
We define several types of feedback property. 
Recall that $W^i$ denotes the $i$th column vector of a matrix $W$.
\begin{definition}
A random model $\{W(k)\}$ has 
\textit{strong feedback property} if there exists 
$\gamma>0$ such that 
\[W_{ii}(k)\geq \gamma\qquad
\hbox{a.s. for all $k\geq 0$ and all $i\in [m]$}.\]
The model has \textit{feedback property} 
if there exists $\gamma>0$ such that 
\[\EXP{W_{ii}(k)W_{ij}(k)}\geq \gamma\, \EXP{W_{ij}(k)},\]
for all $k\geq 0$, and all $i,j\in [m]$ with $i\not =j$. 
The model has \textit{weak feedback property} 
if there exists $\gamma>0$ such that 
\[\EXP{\left( W^i(k)\right)^TW^j(k)}\geq \gamma 
\left(\EXP{W_{ij}(k)}+\EXP{W_{ji}(k)}\right)\]
for all $k\geq 0$, and all $i,j\in [m]$ with $i\not =j$.
 The scalar $\gamma$ is referred to as a {\it feedback constant}.
\end{definition}

While the difference between feedback and strong feedback property
is apparent, the difference between weak feedback and feedback property
may not be so obvious. The following example illustrates the difference
between these concepts.

\begin{example}\label{exam:feed}
Consider the static deterministic chain $\{A(k)\}$:
\[A(k) = A =\left[\begin{array}{ccc}
0&\frac{1}{2}&\frac{1}{2}\\
\frac{1}{2}&0&\frac{1}{2}\\
\frac{1}{2}&\frac{1}{2}&0
\end{array}\right]\qquad\mbox{for $k\geq 0$}.\]
Since $A_{ii}A_{ij}=0$ and $A_{ij}\ne0$
for all $i\ne j$, the model does not have feedback property.
At the same time, since $A_{ij}+A_{ji}=1$ for $i\ne j$, it follows
$(A^i)^T A^j=\frac{1}{4}=\frac{1}{4}(A_{ij}+A_{ji})$. Thus, 
$\{A(k)\}$ has weak feedback property with $\gamma=\frac{1}{4}$.  $\square$
\end{example}

It can be seen that strong feedback property implies feedback property, 
which in turn implies weak feedback property.
The deterministic consensus and averaging models 
in \cite{Tsitsiklis84,Tsitsiklis86,Olshevsky09,Olshevsky09a}
require that the matrices have non-zero diagonal entries and 
uniformly bounded non-zero entries, which is
more restrictive than the strong feedback property.

We next show that the feedback property of a random model implies 
the strong feedback property of its expected model. 

\begin{lemma}\label{lemma:expstrong}
Let a random model $\{W(k)\}$ have feedback property with 
constant $\gamma$. Then, its expected model $\{\EXP{W(k)}\}$ has strong 
feedback property with $\frac{\gamma}{m}.$
\end{lemma}
\begin{proof}
Let the model have feedback property with a constant $\g$. 
Then, by the definition of the feedback property,  
we have $\EXP{W_{ii}(k)W_{ij}(k)}\geq \gamma\, \EXP{W_{ij}(k)}$
for any $k\geq 0$ and $i,j\in [m]$ with $i\not=j$. 
Since $W_{ij}(k)\leq 1$, it follows that
\[\EXP{W_{ii}(k)}\geq \EXP{W_{ii}(k)W_{ij}(k)}\geq \gamma\, \EXP{W_{ij}(k)}\] 
for all $i,j\in[m], i\ne j$. 
The matrices $W(k)$ are stochastic, so that we have 
$\sum_{j=1}^m\EXP{W_{ij}(k)}=1$. Hence, for every $k\ge0$ and $i\in[m]$,
there exists an index $j^*$ (the dependence on $k$ and $i$ is suppressed) 
such that $\EXP{W_{ij^*}(k)}\geq \frac{1}{m}$. If $j^*=i$, then we are done;
otherwise we have $\EXP{W_{ii}(k)}\geq \gamma\ \EXP{W_{ij^*}(k)}\geq 
\frac{\gamma}{m}$
for all $i\in[m]$. Hence, the expected chain has the strong feedback property 
with constant $\frac{\gamma}m$. 
\end{proof}

We now focus on an independent model. We have the following result. 

\begin{lemma}\label{lemma:feedback}
Consider an independent model $\{W(k)\}$. 
Suppose that the model is such that there is an $\eta>0$ 
with the following property: for all $k\geq 0$ and $i,j\in [m]$ with $i\ne j$,
\[\EXP{W_{ij}(k)}>0\quad\implies\quad
\EXP{W_{ii}(k)W_{ij}(k)}\geq \eta,\]
or the following property:\\ for all $k\geq 0$ and $i,j\in [m]$ with $i\ne j$,
\[\EXP{W_{ij}(k)}>0\quad\implies\quad
\EXP{\left(W^i(k)\right)^TW^j(k)}\geq \eta.\] 
Then, respectively, the model has feedback property with constant $\eta$ 
or weak feedback property with constant $\eta/2$. 
\end{lemma}
\begin{proof}
We prove only the case of feedback property, since  
the other case uses the same line of argument. 
If $\EXP{W_{ij}(k)}=0$ for some $k$ and $i,j$ with $i\ne j$, then 
the relation $\EXP{W_{ii}(k)W_{ij}(k)}\geq \eta\, \EXP{W_{ij}(k)}$
is satisfied trivially (with both sides equal to zero).
If $\EXP{W_{ij}(k)}>0$, then by the assumption of the lemma, we have
$\EXP{W_{ij}(k)W_{ii}(k)}\geq \eta$.
Furthermore, since $1\ge W_{ij}(k)$ for all $i,j$ and $k$, it follows
that
$\EXP{W_{ij}(k)W_{ii}(k)}\geq \eta\geq \eta\,\EXP{W_{ij}(k)},$
thus showing that 
the model has feedback property with constant $\eta$. 
\end{proof}

The i.i.d.\ models $\{W(k)\}$ with almost surely positive diagonal
entries $W_{ii}(k)$ have been studied in 
\cite{Zampieri,Tahbaz-Salehi08,Tahbaz-Salehi09}. Such models
have feedback property as seen in the following  corollary.

\begin{corollary}\label{cor:TahbazFeedback} If $\{W(k)\}$ is an 
i.i.d.\ model with almost surely positive diagonal entries,
then the model has feedback property 
with constant\\
\centerline{ $\gamma=
\min_{\{i\ne j\mid \EXP{W_{ij}(k)}>0\}}\ \EXP{W_{ii}(k)W_{ij}(k)}$.}
\end{corollary}
\begin{proof}
Let $\EXP{W_{ij}(k)}>0$ for some $i,j\in[m]$. Since $W_{ii}(k)>0$ {\it a.s.} 
and $W_{ij}(k)\geq 0$, we have $\EXP{W_{ii}(k)W_{ij}(k)}>0$. 
Define $\eta=\min_{\{i\ne j\mid \EXP{W_{ij}(k)}>0\}}\ 
\EXP{W_{ii}(k)W_{ij}(k)}$. 
Since the model is i.i.d., the constant $\eta$ 
is independent of time. Hence, by Lemma~\ref{lemma:feedback} 
it follows that the model has feedback property with constant $\eta$.
\end{proof}

\subsection{Steady State in Expectation}\label{sec:steadystate}
Here, we consider a model with another special property.
Specifically, we discuss a random model $\{W(k)\}$ such that 
its expected chain $\{\EXP{W(k)}\}$ has a common steady state.
\begin{definition} 
A random model has a {\it common steady state in expectation} 
if there is a stochastic vector 
$\pi\in\mathbb{R}^m$ such that $\pi^TE[W(k)]=\pi^T$ for~all~$k$.
\end{definition}
For example, the $m\times m$ matrices that are doubly stochastic
in expectation satisfy the preceding definition with $\pi=\frac{1}{m}\, e$,
such as the matrices arising in a randomized broadcast or gossip over a 
connected (static) network \cite{Aysal09,Boyd06}. 

Consider the function given by
\begin{equation}\label{eqn:lyapunov}
V(x)=\sum_{i=1}^{m}\pi_i\left(x_i-\pi^Tx\right)^2 
\qquad\hbox{for $x\in\re^m$}.
\end{equation}
The function $V(x)$ measures the weighted spread of  
the vector $x$ entries with respect to the weighted average value $\pi^Tx$. 

We at first study the behavior of the weighted averages $\pi^Tx(k)$ 
along the sequence $\{x(k)\}$. 
The main observation is that the random scalar sequence $\{\pi^Tx(k)\}$ is 
a bounded martingale, which leads us to the following result. 
 
\begin{lemma}\label{lemma:martingale}
Let $\{W(k)\}$ be an independent random model with a common steady state 
$\pi$ in expectation. 
Then, the sequence $\{\pi^Tx(k)\}$ converges almost surely for any 
\hbox{$x(0)\in \mathbb{R}^m$}. 
\end{lemma}
\begin{proof}
By the model independency and  $\pi^T\EXP{W(k)}=\pi^T$, 
it  follows that the process $\{\pi^Tx(k)\}$ is a martingale 
with respect to the natural filtration of the process 
for any initial $x(0)\in \mathbb{R}^m$.
Since the matrices $W(k)$ are stochastic, the sequence $\{x(k)\}$ is bounded. 
Thus, $\{\pi^Tx(k)\}$ is a bounded martingale. 
By the martingale convergence theorem 
(see \cite{Billingsley1995}, Theorem~35.5), 
the sequence $\{\pi^Tx(k)\}$ converges a.s. 
\end{proof}

We next characterize the limit of the martingale $\{\pi^Tx(k)\}$. 
Let $\lambda_{i,\pi}\in\mathbb{R}$ be the limit of the martingale
$\{\pi^Tx(k)\}$ for the initial state $x(0)=e_i$, and 
let $\lambda_\pi$ be the vector defined by 
\begin{align}\label{eqn:martingalelimit}
\l_\pi&=(\lambda_{1,\pi},\ldots,\lambda_{m,\pi})\\\nonumber 
&\hbox{with }
\lambda_{i,\pi}=\lim_{k\to\infty}\pi^T W(k)\cdots W(0)e_i 
\ \ \hbox{for $i\in[m]$}.
\end{align}

In the following lemma, we provide some properties of the 
random vector $\lambda_\pi$.
\begin{lemma}\label{lemma:hyperplanes}
Let $\{W(k)\}$ be an independent model with a common steady state $\pi$
in expectation.
Then, the random~vector $\lambda_\pi$ has the following properties:
\begin{enumerate}
\item[(a)] $\lim_{k\to\infty}\pi^Tx(k)=\lambda^T_\pi x(0)$ a.s.\ 
for any $x(0)\in\R^m$. 
\item [(b)] $\lambda_\pi$ is a stochastic vector.
\item [(c)] $\EXP{\lambda_\pi}=\pi$.
\item [(d)] For every $x(0)\in\mathbb{R}^m$, 
the limit points of the sequence $\{x(k)\}$ lie in the random hyperplane 
$\mathcal{H}_{\pi,\lambda_\pi^Tx(0)}
=\{x\in\mathbb{R}^m \mid \pi^Tx=\lambda_\pi^Tx(0)\}$ almost surely.
\end{enumerate}
\end{lemma}
\begin{proof}
By Lemma~\ref{lemma:martingale} and 
the definition of $\lambda_{i\pi}$, we have
$\lim_{k\to\infty}\pi^Tx(k)= \lambda_{i\pi}$ almost surely for initial state
$e_i$.
Using the linearity of the system in~\eqref{eqn:dynsys}, 
we obtain for any $x(0)\in\re^m$,
\begin{align}\nonumber
\lim_{k\to\infty}\pi^Tx(k)&=\sum_{i=1}^m
x_i(0) \left(\lim_{k\to\infty}\pi^T W(k-1)\cdots W(0)e_i \right)\cr 
&=\sum_{i=1}^m x_i(0) \lambda_{i,\pi}= \lambda^T_\pi x(0),
\end{align}
thus showing part (a).

Since $x(k)=W(k-1)\cdots W(0)x(0)$, by part (a) it follows 
\hbox{$\lim_{k\to\infty}\pi^TW(k-1)\cdots W(0)=\lambda^T_\pi$.}
The matrices $W(k)$ and the vector $\pi$ have nonnegative entries
implying that the vector $\lambda_\pi$ also has nonnegative entries. 
By letting $x(0)=e$ and using the stochasticity of $W(k)$, 
we have $x(k)=e$ for all $k$, implying $1=\pi^Tx(k)$ for all $k$. 
Thus, we have
$1=\lim_{k\to\infty}\pi^Tx(k)=\lambda_{\pi}^Tx(0)=\lambda_{\pi}^Te,$
where the second equality holds by part (a).
Hence, $\lambda_{\pi}$ is a stochastic vector. 

To show part (c), we note that by the martingale property of 
the process $\{\pi^Tx(k)\}$, we have $\EXP{\pi^Tx(k)}=\pi^Tx(0)$ 
for all $k\ge0$ and $x(0)\in\mathbb{R}^m$.
By the boundedness of the martingale, we have 
$\lim_{k\to\infty}\EXP{\pi^Tx(k)}=\EXP{\lambda_\pi^T}x(0)$ 
for any $x(0)\in\mathbb{R}^m$. The preceding two relations imply 
$\EXP{\lambda_\pi}=\pi$.  

For part (d), we note that the sequence $\{x(k)\}$ is bounded for every 
$x(0)\in\mathbb{R}^m$ by the stochasticity of $W(k)$;
thus, it 
has accumulation points.
By part (a), each accumulation point $x^*$ of the sequence
satisfies $\pi^T x^* =\lambda_\pi^Tx(0)$ a.s.
\end{proof}

We now focus on the sequence $\{V(x(k))\}$. We show that it
is a convergent supermartingale, which indicates that 
$V(x)$ is a stochastic Lyapunov function for the random system
in~\eqref{eqn:dynsys}.

\begin{theorem}\label{thm:supermartingale}
Let the random model $\{W(k)\}$ be independent with a common steady 
state~$\pi$ in expectation. Then, we almost surely have for all $k\ge0$, 
\begin{align}\label{eqn:supermartingale}
&\EXP{V(x(k+1))\mid x(k)}\\\nonumber
&\qquad\le V(x(k))- \sum_{i<j}H_{ij}(k)\left(x_i(k)-x_j(k)\right)^2,
\end{align}
where $H(k)=\EXP{W^T(k)DW(k)}$.
Furthermore, $\{V(x(k))\}$ converges 
almost surely.
\end{theorem}

\begin{proof} 
By using $D=diag(\pi)$, from the definition of the function $V(x)$
in~\eqref{eqn:lyapunov} we have
\[V(x)=x^T(I-\pi e^T)D(I-e\pi^T)x=x^T(D-\pi\pi^T)x,\]
where the second equality is obtained by using 
$e^TD=\pi^T$, $De=\pi$, and $\pi^Te=1$.
In view of $x(k+1)=W(k)x(k)$, it follows that for all $k\ge0$,
\begin{align}\label{eqn:vxk}
V(x(k+1))&=x(k+1)^T(D-\pi\pi^T)x(k+1) 
\cr &= x(k)W(k)^T(D-\pi\pi^T)W(k)x(k).
\end{align}
Since the model is independent, by taking the expectation 
conditioned on $x(k)$, we obtain
\[\EXP{V(x(k+1)) \mid x(k)}= x(k)^T\EXP{W(k)^T(D-\pi\pi^T)W(k)}x(k),\] 
almost surely for all $k\ge0$.
Since $H(k)=\EXP{W^T(k)DW(k)}$, we further have
\begin{align*}
&\EXP{V(x(k+1))\mid x(k)}\cr 
&=x(k)^TH(k)x(k) 
-\EXP{\left(\pi^TW(k) x(k)\right)^2\mid x(k)}\cr
& \le x(k)^TH(k)x(k) 
-\left(\EXP{\pi^TW(k)x(k)\mid x(k)}\right)^2,
\end{align*}
where the inequality follows by Jensen's inequality 
(see \cite{Durrett}, page 225)
and the convexity of the function $s\mapsto s^2$.
The expected matrices $\EXP{W(k)}$ have the same steady state $\pi$, implying 
that almost surely for all $k\geq 0$,
\[\EXP{\pi^TW(k)x(k)\mid x(k)}=\pi^T\EXP{W(k)}x(k)
=\pi^T x(k).\]
By combining the preceding two relations, we see that 
almost surely for all $k\ge0$,
\[\EXP{V(x(k+1)) \mid x(k)} \le x(k)^T\left( H(k)-\pi\pi^T \right) x(k).\]
By adding and subtracting $x(k)^TDx(k)$ to the right hand side of the preceding
relation and using~(\ref{eqn:vxk}), we obtain almost surely
\begin{align}\label{eqn:supermart}
&\EXP{V(x(k+1))\mid x(k)}\\\nonumber  
& \le x(k)^T\left( H(k)-D\right) x(k)
 + V(x(k))\ \ \hbox{ for all $k\ge0$}.
\end{align}

Now, we show that $x(k)^T\left( H(k)-D\right) x(k) 
=\sum_{i<j}H_{ij}(k) (x_i(k)-x_j(k))^2$.
By the definition of $H(k)$ we have
$H_{ij}(k)=\sum_{\ell=1}^{m}\EXP{\pi_\ell W_{\ell i}(k) W_{\ell j}(k)}$,
so that \hbox{for $i\in[m]$},
\begin{align*}
\sum_{j=1\atop j\not=i}^m H_{ij}(k)
&=\EXP{\sum_{j=1\atop j\not=i}^m
\sum_{\ell=1}^{m}\pi_\ell W_{\ell i}(k) W_{\ell j}(k)}
\cr 
&=\EXP{\sum_{\ell=1}^{m}\pi_\ell W_{\ell i}(k)
\sum_{j=1\atop j\not=i}^m W_{\ell j}(k)}.
\end{align*}
Since $W(k)$ is stochastic, we have
$\sum_{j=1, j\not=i}^m W_{\ell j}(k)=1-W_{\ell i}(k)$,
implying that
\begin{align*}
\sum_{j=1\atop j\not=i}^m H_{ij}(k)
&=\EXP{\sum_{\ell=1}^{m}\pi_\ell W_{\ell i}(k)}
-\EXP{\sum_{\ell=1}^{m}\pi_\ell W^2_{\ell i}(k)}\cr 
&=\pi_i-H_{ii}(k),
\end{align*}
where the last equality follows from $\EXP{\pi^TW(k)}=\pi^T$. 
Since $D=diag(\pi)$, the preceding relation yields
\begin{equation}\label{eqn:uno}
H_{ii}(k)=D_{ii}-\sum_{j=1\atop j\not=i}^m H_{ij}(k).
\end{equation} 
Therefore, for any $x\in\mathbb{R}^m$, we have
$x^T H(k)x=\sum_{i=1}^m x_i\sum_{j=1}^m H_{ij}(k)x_j$,
which can be further written as
\begin{align*}
x^T H(k)x 
&=\sum_{i=1}^m x_i \sum_{j=1\atop j\ne i}^m H_{ij}(k)x_j +
\sum_{i=1}^m x_i H_{ii}(k) x_i\cr 
&=\sum_{i=1}^m x_i \sum_{j=1\atop j\ne i}^m H_{ij}(k)(x_j -x_i)
+\sum_{i=1}^m x_i D_{ii} x_i,
\end{align*}
where the last equality follows from relation (\ref{eqn:uno}).
The matrix $H(k)=\EXP{W(k)^TDW(k)}$ is symmetric, so that
\[\sum_{i=1}^m x_i \sum_{j=1\atop j\ne i}^m H_{ij}(k)(x_j -x_i)
=-\sum_{i<j}H_{ij}(k)(x_i -x_j)^2,\]
implying that $x^T H(k)x =-\sum_{i<j} H_{ij}(k)(x_i -x_j)^2+x^TD x.$
Therefore, we have
\[x(k)^T \left(H(k)-D\right)x(k)=
-\sum_{i<j}H_{ij}(k)\left(x_i(k)-x_j(k)\right)^2.\]
By combining the preceding relation with~\eqref{eqn:supermart},
we conclude that relation~\eqref{eqn:supermartingale} holds a.s.\ 
for all $k\geq 0$. 

Since each $W(k)$ is a stochastic matrix and $\pi$ is stochastic vector, 
the matrix $H(k)$ has nonnegative entries for all $k$.
Hence, from the preceding relation it follows that $\{V(x(k))\}$ is a
supermartingale. The convergence of $\{V(x(k))\}$ 
follows straightforwardly from the nonnegative supermartingale convergence
(see \cite{Durrett}, (2.11) Corollary, page 236.)
\end{proof}

We conclude this section with another result 
for the weighted distance function $V(x)$. This result plays
crucial role in establishing our result in Section~\ref{sec:sufficient}.
\begin{lemma}\label{lemma:fundament}
Let $\pi\in\mathbb{R}^m$ be a stochastic vector, and let $x\in\mathbb{R}^m$
be such that $x_1\leq\cdots\leq x_m$. Then, 
we have
\begin{align*}
\frac{1}{(m-1)^2}\,V(x)\le \frac{1}{x_m-x_1}\,\sum_{i=1}^{m-1}(x_{i+1}-x_i)^3. 
\end{align*}
\end{lemma}

\begin{proof}
We establish two relations
\begin{align}
\frac{1}{m-1}\,V(x)&\le \sum_{i=1}^{m-1}(x_{i+1}-x_i)^2,\label{eqn:one}\\
\sum_{i=1}^{m-1}\frac{(x_{i+1}-x_i)^2}{m-1}
&\le \sum_{i=1}^{m-1}\frac{(x_{i+1}-x_i)^3}{x_m-x_1}.\label{eqn:two}
\end{align}
Observe that the desired result follows from~\eqref{eqn:one}--\eqref{eqn:two}.

We now show relation~\eqref{eqn:one}.
We have $x_i\le x_m$ for all $i$.
Since $\pi$ is stochastic, we also have $x_m\ge \pi^Tx\ge x_1$. 
Thus, $V(x)=\sum_{i=1}^{m}\pi_i(x_i-\pi^Tx)^2\leq(x_m-x_1)^2$. 
By writing $x_m-x_1=\sum_{i=1}^{m-1} (x_{i+1}-x_i)$, we obtain  
\begin{align*}
(x_m-x_1)^2 
&=(m-1)^2\left(\frac{1}{m-1}\sum_{i=1}^{m-1}(x_{i+1}-x_i)\right)^2 \cr 
&\leq (m-1)\sum_{i=1}^{m-1}(x_{i+1}-x_i)^2,
\end{align*}
where the last inequality holds by the convexity of the function 
$s\mapsto s^2$. Using $V(x)\leq (x_m-x_1)^2$ and the preceding relation 
we obtain relation~\eqref{eqn:one}.

To prove relation~\eqref{eqn:two}, we write 
$(x_m-x_1)\sum_{i=1}^{m-1}(x_{i+1}-x_i)^2$
as $\sum_{j=1}^{m-1}(x_{j+1}-x_j)\sum_{i=1}^{m-1}(x_{i+1}-x_i)^2$ ,
which is equal to $\sum_{j=1}^{m-1}(x_{j+1}-x_j)^3+\Delta$ with
$\Delta$ given by
\[\sum_{j<i}
\left((x_{j+1}-x_j)(x_{i+1}-x_i)^2 + (x_{i+1}-x_i)(x_{j+1}-x_j)^2\right).\]
To estimate $\Delta$, we
consider scalars $\alpha\ge0$ and $\beta\ge0$, 
and let $u=(\alpha,\beta)$ and $v=(\beta^2,\alpha^2)$. 
Then, by H\"older's inequality with $p=3$, $q=\frac{3}2$, we have 
$u^T v\leq \|u\|_p\,\|v\|_q$, where $\|\cdot\|_p$ is the $p$-norm. Hence, 
\begin{equation}\label{eqn:coolineq}
\alpha\beta^2+\beta\alpha^2\le \left(\a^3+\beta^3\right)^{\frac{1}{3}}
\left(\beta^3+\a^3\right)^{\frac{2}{3}}=\alpha^3+\beta^3. 
\end{equation}
By using~\eqref{eqn:coolineq} 
with $\alpha_j=(x_{j+1}-x_j)$ and $\beta_i=(x_{i+1}-x_i)$ 
for different indices $j$ and $i$, $1\le j<i\leq m-1$, we obtain
\begin{align*}
&(x_m-x_1)\sum_{i=1}^{m-1}(x_{i+1}-x_i)^2\cr 
&\le \sum_{j=1}^{m-1}(x_{j+1}-x_j)^3 +
\sum_{j<i} \left((x_{j+1}-x_j)^3 + (x_{i+1}-x_i)^3\right)\cr
& = (m-1)\,\sum_{i=1}^{m-1}(x_{i+1}-x_i)^3,
\end{align*}
which completes the proof.
\end{proof}

\section{Model with Infinite Flow Property}\label{sec:infiniteflow}
We consider an independent random  model 
with infinite flow. We show that this property, together with weak feedback
and a common steady state in expectation, 
is necessary and sufficient for almost sure ergodicity. 
Moreover, we establish that the ergodicity 
of the model is equivalent to the ergodicity of the expected model.

\subsection{Preliminary Result}\label{sec:prelim}
We now provide an important relation that we use later on in 
Section~\ref{sec:sufficient}.

\begin{lemma}\label{lemma:essential}
Let $\{A(k)\}\subset\mathbb{S}^m$ 
and $z(k+1)=A(k)z(k)$ for all $k\ge0$ and some 
$z(0)\in\mathbb{R}^m.$
Let $\sigma$ be a permutation of the index set
$[m]$ corresponding to the nondecreasing ordering of 
the entries $z_\ell(0)$, i.e., 
$\sigma$ is a permutation on $[m]$ 
such that $z_{\sigma_1}(0)\leq \cdots\leq z_{\sigma_m}(0)$.
Also, let $T\geq 1$ be such that 
\begin{equation}\label{eqn:flowbyT}
\sum_{k=0}^{T-1} A_S(k)\geq \delta   
\qquad\hbox{for every $S\subset[m]$},
\end{equation} 
where $\delta\in(0,1)$ is arbitrary. 
Then, we have
\begin{align*}
&\sum_{k=0}^{T-1} \sum_{i<j} 
\left(A_{ij}(k)+A_{ji}(k)\right) (z_j(k)-z_i(k))^2\cr 
&\geq \frac{\delta(1-\delta)^2}{z_{\sigma_{m}}(0)-z_{\sigma_1}(0)}
\sum_{i=1}^{m-1}(z_{\sigma_{i+1}}(0)-z_{\sigma_{i}}(0))^3.
\end{align*}
\end{lemma}

\begin{proof}
Relation~\eqref{eqn:flowbyT} holds for any
nontrivial set $S\subset[m]$. Hence,
without loss of generality we may assume that the 
permutation $\sigma$ is identity (otherwise we will relabel the indices
of the entries in $z(0)$ and update the matrices accordingly).
Thus, we have $z_1(0)\leq\cdots\leq z_m(0)$. 
For each $\ell=1,\ldots,m-1$, let $S_\ell=\{1,\ldots,\ell\}$  and define time 
$t_\ell\ge 1$, as follows:
\begin{align*}
t_\ell=\argmin_{t\geq 1}\left\{
\sum_{k=0}^{t-1} \  A_{S_\ell}(k) 
\geq \delta\, \frac{z_{\ell+1}(0)-z_{\ell}(0)}{z_m(0)-z_1(0)}\right\}.
\end{align*}
Since the entries of $z(0)$ are nondecreasing, we have
$\delta\frac{z_{\ell+1}(0)-z_{\ell}(0)}{(z_m(0)-z_1(0))}\le
\delta$ for all $\ell=1,\ldots,m-1$. Thus,
by relation~\eqref{eqn:flowbyT}, the time $t_\ell\geq 1$ exists 
and $t_\ell\le T$ for each~$\ell$.  

We next estimate $z_j(k)-z_i(k)$ for all $i<j$ and any time $k=0,\cdots,T-1$.
For this, we introduce for $0\le k\le T-1$ and $i<j$
the index sets $a_{ij}(k)\subset[m]$, as follows:
\begin{align*}
a_{ij}(k)=\{\ell\in[m]\mid k\leq t_\ell-1,\ \ell\ge i,\ \ell+1\leq j\}.
\end{align*}
Let $k\le t_\ell-1$ for some $\ell$. 
Since $S_\ell=\{1,\ldots,\ell\}$,
we have $i\in S_\ell$ and $j\in \bar S_\ell$.
Thus, by Lemma~\ref{lemma:flowmaxminS} we have for any $k\ge1$,
\[z_{i}(k)\le \max_{s\in S_\ell}z_s(0) 
+(z_m(0)-z_1(0)) \sum_{\tau=0}^{k-1} A_{S_\ell}(\tau),\]
\[z_{j}(k)\ge \min_{r\in\bar S_\ell}z_r(0) 
- (z_m(0)-z_1(0)) \sum_{\tau=0}^{k-1} A_{S_\ell}(\tau).\]
Furthermore, $\max_{s\in S_\ell}z_s(0)=z_{\ell}(0)$
and $\min_{r\in\bar S_\ell}z_r(0)=z_{l+1}(0)$ since $S_\ell=\{1,\ldots,\ell\}$ 
and $z_1(0)\le \cdots\le z_m(0)$. Thus, it follows
\[z_{i}(k)-z_{\ell}(0)\leq (z_m(0)-z_1(0))\sum_{\tau=0}^{k-1}A_{S_\ell}(\tau),
\]
\[z_{\ell+1}(0)- z_{j}(k)
\le (z_m(0)-z_1(0))\sum_{\tau=0}^{k-1}A_{S_\ell}(\tau).\]
By the definition of time $t_\ell$, we 
have $(z_m(0)-z_1(0))\sum_{\tau=0}^{k-1}A_{S_\ell}(\tau)< \delta
\left(z_{\ell+1}(0)-z_\ell(0)\right)$ for $k\le t_\ell-1$.
Hence, by using this and the definition of $a_{ij}(k)$, 
for any $\ell\in a_{ij}(k)$ we have
\begin{align}\label{eqn:zi}
 z_i(k)- z_\ell(0) &\le \delta(z_{\ell+1}(0)-z_\ell(0)),\\\label{eqn:zj}
z_{\ell+1}(0)-z_{j}(k) &\le \delta(z_{\ell+1}(0)-z_\ell(0)).
\end{align}

Now suppose that $a_{ij}(k)=\{\ell_1,\ldots,\ell_r\}$
for some $r\le m-1$ and $\ell_1\le\cdots\le \ell_r$.
By choosing $\ell=\ell_1$ in~\eqref{eqn:zi} and $\ell=\ell_r$ 
in~\eqref{eqn:zj}, and by letting $\alpha_i=z_{i+1}(0)-z_i(0)$,
we obtain 
\[z_j(k)-z_i(k)\geq 
z_{\ell_r+1}(0)-z_{\ell_1}(0)-\delta(\alpha_{\ell_r}+\alpha_{\ell_1}).\]
Since $z_i(0)\le z_{i+1}(0)$ for all $i=1,\ldots,m-1$, we have
$z_{\ell_1}(0)\leq z_{\ell_1+1}(0)\leq \cdots\le z_{\ell_r}(0)
\leq z_{\ell_r+1}(0)$,
which combined with the preceding relation yields
$z_j(k)-z_i(k)
\geq \sum_{\xi=1}^{r}(z_{\ell_{\xi}+1}(0)-z_{\ell_{\xi}}(0))
-\delta(\alpha_{\ell_r}+\alpha_{\ell_1})$. 
Using $\alpha_i=z_{i+1}(0)-z_i(0)$ and $a_{ij}(k)=\{\ell_1,\ldots,\ell_r\}$, 
we further have 
\begin{align}\label{eqn:zestim}
z_j(k)-z_i(k)&\ge\sum_{\xi=1}^{r}\alpha_{\ell_\xi}
-\delta(\alpha_{\ell_r}+\alpha_{\ell_1})
\\\nonumber &\geq (1-\delta)\sum_{\xi=1}^{r-1}\alpha_{\ell_\xi}
=(1-\delta)\sum_{\ell\in a_{ij}(k)}\alpha_{\ell}.
\end{align}

By Eq.~\ref{eqn:zestim}, it follows that
\begin{align*}
&\sum_{i<j} 
\left(A_{ij}(k) + A_{ji}(k)\right) (z_j(k)-z_i(k))^2\cr 
&\quad \geq(1-\delta)^2\sum_{i<j}  \left(A_{ij}(k) + A_{ji}(k)\right)
\left(\sum_{\ell\in a_{ij}(k)}\alpha_\ell\right)^2\cr
&\quad \geq(1-\delta)^2\sum_{i<j} \left(A_{ij}(k) + A_{ji}(k)\right)
\left(\sum_{\ell\in a_{ij}(k)}\alpha_\ell^2\right),
\end{align*}
where the last inequality holds by $\alpha_\ell\geq 0$.
In the last term in the preceding relation,
the coefficient of $\alpha^2_\ell$
is equal to $(1-\delta)^2A_{S_\ell}(k)$. Furthermore,
by the definition of $a_{ij}(k)$, we have $\ell\in a_{ij}(k)$ only when
$k\le t_\ell-1$. Therefore,
\begin{align*}
&\sum_{i<j}\left(A_{ij}(k) + A_{ji}(k)\right)(z_j(k)-z_i(k))^2\cr 
&\quad \ge(1-\delta)^2\sum_{i<j} \left(A_{ij}(k) + A_{ji}(k)\right)
\left(\sum_{\ell\in a_{ij}(k)}\alpha^2_\ell\right)\cr
&\quad  
=(1-\delta)^2\sum_{\{\ell \mid k\leq t_\ell-1\}}A_{S_\ell}(k)\alpha^2_\ell.
\end{align*}
Summing these relations over $k=0,\ldots,T-1$, we obtain
\begin{align*}
&\sum_{k=0}^{T-1}\sum_{i<j}\left(A_{ij}(k) 
+ A_{ji}(k)\right)(z_j(k)-z_i(k))^2\cr 
&\quad \geq (1-\delta)^2\sum_{k=0}^{T-1}\sum_{\{\ell \mid k\leq t_\ell-1\}}
A_{S_\ell}(k)\alpha^2_\ell \cr
&\quad\geq(1-\delta)^2\sum_{\ell=1}^{m-1}
\left(\sum_{k=0}^{t_\ell-1}A_{S_\ell}(k)\right)\alpha_\ell^2,\end{align*}
where the last inequality follows by exchanging the order of summation.
By the definition of $t_\ell$ and using
$\a_\ell=z_{\ell+1}(0)-z_\ell(0)$, 
we have $\sum_{k=0}^{t_\ell-1}A_{S_\ell}(k)
\ge\frac{\delta\a_\ell}{z_m(0)-z_1(0)}$, implying 
\begin{align*}
&\sum_{k=0}^{T-1}\sum_{i<j}\left(A_{ij}(k) 
+ A_{ji}(k)\right)(z_j(k)-z_i(k))^2\cr 
&\qquad =\delta(1-\delta)^2\,
\sum_{\ell=1}^{m-1}\frac{\alpha_\ell^3}{z_m(0)-z_1(0)}.
\end{align*}
\end{proof}

\subsection{Sufficient Conditions for Ergodicity}\label{sec:sufficient}
We now establish one of our main results for
independent random models with infinite flow property
and having a common vector in expectation and weak feedback property. 

Let $t_0=0$ and for any $q\geq 1$, let 
\begin{align}\label{eqn:qtime}
t_q=\argmin_{t\geq t_{q-1}+1}
\prob\left(\min_{S\subset[m]}\ \sum_{k=t_{q-1}}^{t-1}W_{S}(k)
\geq \delta\right)\geq \epsilon,
\end{align}
where $\epsilon,\delta\in (0,1)$ are arbitrary. Define 
\begin{align}\label{eqn:qset}
\mathscr{A}_q=\left\{\omega \,\Big|\,\min_{S\subset[m]}
\sum_{k=t_q}^{t_{q+1}-1}W_{S}(k)(\o)\geq \delta\right\}
\quad\hbox{for $q\ge 0$}.
\end{align}
Since the model has infinite flow property,
the infinite flow event $\mathscr{F}$ occurs a.s.
Therefore, the time $t_q$ is finite for all $q$. 

We next show that either the infinite flow or expected infinite flow
property is sufficient for the ergodicity of the model. 

\begin{theorem}\label{thm:sufficient}
({\it Sufficient Ergodicity Condition}) \ 
Let $\{W(k)\}$ be an independent random model with
a common steady state $\pi>0$ in expectation and weak feedback property. 
Also, let the model have either infinite flow 
or expected infinite flow property. Then, the model is ergodic. 
In particular, $\lim_{k\to\infty} x_i(k)=\lambda_\pi^Tx(0)$ almost surely 
for all $i\in[m],$
where $\lambda_\pi$ is the random vector of Eq.~\eqref{eqn:martingalelimit}.  
\end{theorem}

\begin{proof} 
Assume that the model has infinite flow property.
Let $x(0)\in\mathbb{R}^m$ be arbitrary initial state. 
Let us denote the (random) ordering of the entries of the vector 
$x(t_q)$ by $\eta^q$ for all $q$. Thus, at time $t_q$, 
we have $x_{\eta^q_1}(t_q)\leq\cdots\leq x_{\eta^q_m}(t_q)$. 

Now, let $q\ge0$ be arbitrary and fixed, and consider 
the set $\mathscr{A}_q$ in~\eqref{eqn:qset}.
By the definition of $\mathscr{A}_q$, 
we have $\sum_{k=t_q}^{t_{q+1}-1}W_{S}(k)(\omega)\geq \delta$ 
for any $S\subset[m]$ and $\omega\in \mathscr{A}_q$. 
Thus, by Lemma \ref{lemma:essential}, we obtain
for any $\omega\in \mathscr{A}_q$, 
\begin{align*}
\sum_{t=t_q}^{t_{q+1}-1}\sum_{i<j}\   
&\left(W_{ij}(t)+W_{ji}(t)\right)(x_{i}(t)-x_j(t))^2(\o)\cr 
&\geq \frac{\delta(1-\delta)^2}{d(t_q)(\o)}
\sum_{\ell=1}^{m-1}(x_{\eta^q_{\ell+1}}(t_q)-x_{\eta^q_\ell}(t_q))^3(\o)\cr 
&\geq \frac{\delta(1-\delta)^2}{(m-1)^2}\,V(x(t_q))(\o),
\end{align*}
where $d(t_q)=x_{\eta^q_{m}}(t_q)-x_{\eta^q_1}(t_q)$ and the last inequality 
follows by Lemma~\ref{lemma:fundament}. We can compactly write the
inequality~as:
\begin{align}\label{eqn:resultofessent}
&\sum_{t=t_q}^{t_{q+1}-1}\ \sum_{i<j}  
\left(W_{ij}(t)+W_{ji}(t)\right)(x_{i}(t)-x_j(t))^2\nonumber\\ 
&\qquad\geq \frac{\delta(1-\delta)^2}{(m-1)^2}\ 
V(x(t_q)){\it 1}_{\mathscr{A}_q}, 
\end{align}
where ${\it 1}_{\mathscr{A}_q}$ is the indicator function of the event 
${\mathscr{A}_q}$. 

Observe that $x(t)$ and $W(t)$ are independent since the model is independent. 
Therefore, by Theorem~\ref{thm:supermartingale},
we have
\begin{align*}
&\EXP{V(x(t_{q+1}))-V(x(t_{q}))}\cr
& \leq 
-\sum_{t=t_q}^{t_q-1}\sum_{i<j} H_{ij}(t)\EXP{(x_i(t)-x_j(t))^2},
\end{align*}
with $H(k)=\EXP{W^T(k)DW(k)}$ and $D=diag(\pi).$
Let $\pi_{\min}=\min_{\ell\in[m]} \pi_\ell$ and note that $\pi_{\min}>0$
by $\pi>0$. Thus, 
\begin{align*}
H_{ij}(t)&=\EXP{(W^i(t))^TDW^j(t)}\cr 
&\ge \pi_{\min}\,\EXP{(W^i(t))^T W^j(t)}\cr &
\ge\pi_{\min}\,\g\,\left(\EXP{W_{ij}(t)}+\EXP{W_{ji}(t)}\right).
\end{align*}
Therefore, 
\begin{align*}
&\EXP{V(x(t_{q+1}))}-\EXP{V(x(t_{q}))}\cr
&\leq-\pi_{\min}\, \gamma\,
\EXP{\sum_{t=t_q}^{t_q-1}\sum_{i<j}  
\left(W_{ij}(t) + W_{ji}(t)\right)(x_i(t)-x_j(t))^2}.
\end{align*}
Further, by using relation~\eqref{eqn:resultofessent}, we obtain 
\begin{align*}
&\EXP{V(x(t_{q+1}))}-\EXP{V(x(t_{q}))}
\cr 
&\qquad\leq-\pi_{\min}\,\gamma\ 
\EXP{\frac{\delta(1-\delta)^2}{(m-1)^2}\,{\it 1}_{A_q}V(x(t_q))}\cr 
&\qquad\leq -\frac{\epsilon\delta(1-\delta)^2\gamma\pi_{\min}}{(m-1)^2}
\,\EXP{V(x(t_q))},
\end{align*}
where the last inequality follows by 
$\prob(\mathscr{A}_q)\geq \epsilon$, and the fact that 
${\it 1}_{\mathscr{A}_q}$ and $V(x(t_q))$ are independent
(since $x(t_q)$ depends on information prior to time $t_q$
and the set $\mathscr{A}_q$ relies on information at time $t_q$ and later). 
Hence, it follows
\[\EXP{V(x(t_{q+1}))}\leq 
\left(1-\frac{\epsilon\delta(1-\delta)^2\gamma\pi_{\min}}{(m-1)^2}\right)
\EXP{V(x(t_q))}.\]
Therefore, for arbitrary $q\ge0$ we have
\[\EXP{V(x(t_{q}))}\leq 
\left(1-\frac{\epsilon\delta(1-\delta)^2\gamma\pi_{\min}}{(m-1)^2}\right)^q
\EXP{V(x(0))},\] 
implying that $\sum_{q=0}^{\infty}\EXP{V(x(t_q))}<\infty$.
In view of the nonnegativity of $V(x)$, 
by the monotone convergence theorem 
(\cite{Billingsley1995}, Theorem 16.6) 
it follows
$\EXP{\sum_{q=0}^\infty V(x(t_q))}<\infty$,
implying $\lim_{q\rightarrow\infty}V(x(t_q))=0$ {\it a.s.}
According to Theorem~\ref{thm:supermartingale}, the sequence 
$\{V(x(k))\}$ is convergent, 
which together with the preceding relation implies that 
$\lim_{k\rightarrow\infty}V(x(k))=0$ {\it a.s.}

To show the ergodicity of the model, we note that by the convergence 
result for the martingale $\{\pi^Tx(k)\}$ in Lemma~\ref{lemma:hyperplanes}(a),
we have $\lim_{k\to\infty}\pi^Tx(k)=\lambda_\pi^Tx(0)$ a.s., 
where the random vector $\lambda_\pi$ is given by 
Eq.~\eqref{eqn:martingalelimit}.
Now, using $\pi^Tx(k)\to\lambda_\pi^Tx(0)$ and the fact that all norms
in $\mathbb{R}^m$ are equivalent, we obtain 
$\lim_{k\to\infty}|x_i(k)-\lambda_\pi^Tx(0)|=0$
{\it a.s.} for all $i\in[m],$ thus showing the ergodicity of the model.

Assume now that $\{W(k)\}$ has expected infinite flow,
i.e., $\sum_{k=0}^\infty\, \EXP{W_S(k)}=\infty$ 
for any nontrivial $S\subset[m]$.
By Theorem~\ref{thm:expectedflow}, $\{W(k)\}$ has infinite flow property if 
and only if it has expected infinite
flow, and the result follows by the preceding case.
\end{proof}

\subsection{Necessary and Sufficient Conditions for Ergodicity}
\label{sec:infiniteflowthm}
Here, we provide the central result of this paper. The result
establishes necessary and sufficient conditions for ergodicity of random models
with weak feedback property and a common steady state $\pi>0$ in expectation. 
The conditions are reliant on infinite flow, and guarantee 
that the ergodicity of the model is equivalent to the ergodicity of
the expected model. The result emerges as an outcome of several important
results that we have developed so far. In particular, we combine
the result $\mathscr{E}\subseteq\mathscr{F}$ 
stating that the ergodicity event is always contained in the
infinite flow event (Theorem~\ref{thm:infiniteflow}),
the deterministic characterization of the infinite flow of 
Theorem~\ref{thm:expectedflow}, and the sufficient conditions of 
Theorem~\ref{thm:sufficient}. We also make use of Theorem~\ref{thm:equivalency}
providing conditions for equivalence of the ergodicity of the chain 
and the expected chain.

\begin{theorem}\label{thm:infflowthm} ({\it Infinite Flow Theorem}) \ 
Let the random model $\{W(k)\}$ be independent, and 
have a common steady state $\pi>0$ in expectation and
weak feedback property. 
Then, the following conditions are equivalent: 
\begin{enumerate}
\item[(a)] The model is ergodic.
\item[(b)] The model has infinite flow property.
\item[(c)] The expected model has infinite flow property. 
\item[(d)] The expected model is ergodic.
\end{enumerate}
\end{theorem}
\begin{proof}
First, we establish that parts (a), (b) and (c) are equivalent
by showing that (a) $\implies$ (b) $\iff$ (c) $\implies$ (a).
In particular, by Theorem~\ref{thm:infiniteflow} we have 
$\mathscr{E}\subseteq \mathscr{F}$, showing that 
(a) $\implies$ (b). 
By Theorem \ref{thm:expectedflow}, parts (b) and (c) are equivalent.
By Theorem~\ref{thm:sufficient}, part (c) implies part (a).
Now, we prove (a) $\iff$ (d). Since (a) $\iff$ (b), we have 
$\mathscr{E}=\mathscr{F}$ a.s. Hence, by 
Theorem~\ref{thm:equivalency}, the parts (a) and (d) are equivalent.
\end{proof}

The infinite flow theorem combined with the deterministic characterization
of the infinite flow model of Theorem~\ref{thm:expectedflow} 
leads us to the following result.

\begin{corollary}\label{cor:detergflow}
Let $\{A(k)\}\subset\mathbb{S}^m$ be a deterministic model that 
has a common steady state vector $\pi>0$ and weak feedback property.
Then, the chain $\{A(k)\}$ is ergodic if and only if 
$\sum_{k=0}^\infty A_S(k)=\infty$
for every nontrivial $S\subset[m]$. 
\end{corollary}

Under the conditions of Theorem~\ref{thm:infflowthm}, the model admits 
consensus, which follows directly from 
relation $\mathscr{E}\subseteq\mathscr{C}$.

\begin{corollary}\label{cor:consexpflow}
Let the  assumptions of Theorem~\ref{thm:infflowthm} hold. 
Then, the model admits consensus.
\end{corollary}

The infinite flow theorem establishes the equivalence between the ergodicity 
of a chain and the expected chain for a class of independent random models.
The central role in this result is played by the infinite flow 
and its equivalent deterministic characterization. Another crucial
result is the interplay between the ergodicity and infinite flow 
of Theorem~\ref{thm:equivalency}
yielding the equivalence between ergodicity of the chain and the expected
chain. The following two examples are provided to illustrate some
straightforward applications of the infinite flow theorem.

\noindent\textit{i.i.d.\ Models}.\quad
Consider an i.i.d.\ model $\{W(k)\}$. 
Then, the expected matrix $\bar{W}=\EXP{W(k)}$ is independent of $k$. 
Since $\bar{W}$ is stochastic, 
we have $\pi^T\bar{W}=\pi^T$ for a stochastic vector $\pi\ge0$. 
Therefore, an i.i.d.\ model is an independent model
with a common steady state $\pi$ in expectation.  

In~\cite{Tahbaz-Salehi08}, it is shown that for the class of i.i.d.\ models 
that have a.s.\ positive diagonal entries, 
the ergodicity of the expected model and the ergodicity of the original model 
are equivalent. The application of this result is reliant on 
the condition of the a.s.\ positive diagonal entries, which 
implies that the model has feedback property, as 
shown in~Corollary~\ref{cor:TahbazFeedback}. This property, however, is
stronger than weak feedback property. At the same time, no requirement
on the steady state vector is needed.

The application of the infinite flow theorem to the i.i.d.\ case would require
weak feedback property and the existence of a steady state vector $\pi>0$. 
Thus, there is a tradeoff in the 
conditions for the ergodicity provided by the infinite flow theorem
and those given in \cite{Tahbaz-Salehi08}. 
To further illustrate the difference in the conditions,
we consider the homogeneous deterministic model $\{A(k)\}$ of 
Example~\ref{exam:feed}. The model $\{A(k)\}$ has weak feedback property
and the steady state vector $\pi=\frac{1}{3}e$, so the ergodicity of 
the model can be deduced from the infinite flow theorem.
At the same time, as seen in Example~\ref{exam:feed}, the model does not 
have positive diagonal entries and, therefore, the ergodicity of the model
cannot be deduced from the results 
in~\cite{Tahbaz-Salehi08,Tahbaz-Salehi09}. 
In the light of this, the infinite flow
theorem provides conditions for ergodicity that 
complement the conditions of \cite{Tahbaz-Salehi08,Tahbaz-Salehi09}. 

\noindent
\textit{Gossip Algorithms on Time-varying Networks}.\quad
As another application of the infinite flow theorem, we consider
an extension of the standard gossip algorithm to time-varying networks.
In particular, the gossip algorithm originally proposed in 
\cite{Boyd05,Boyd06} is for static networks. Here,
we give a sufficient condition for the convergence of a gossip algorithm for 
networks with time-changing topology. 
Consider a network of $m$ agents viewed as nodes of a 
graph with the node set $[m]$. Suppose that each agent has a private scalar 
value $x_i(0)$ at time $k=0$. Now, let the interactions of the agents be 
random at nonnegative integer valued time instances 
$k$ as follows: At any time $k\geq 1$, two different agents 
$i,j\in [m]$ wake up with probability $P_{ij}(k)$, where  
$P_{ij}(k)=P_{ji}(k)$ and $\sum_{i<j}P_{ij}(k)=1$. Then, they set their values 
to the average of their current values, i.e., 
$x_i(k)=x_j(k)=\frac{1}{2}(x_i(k-1)+x_j(k-1))$. 
The choices of the pairs $\{i,j\}$
of interacting agents at different time instances are independent. 

Based on the agent interaction model, define 
the independent random model $\{W(k)\}$ by:
\begin{align}\label{eqn:gossip}
W(k)=I-\frac{1}{2}(e_i-e_j)(e_i-e_j)^T\ \mbox{with prob.\ $P_{ij}(k)$}.
\end{align}
Then, the dynamic system~\eqref{eqn:dynsys} driven by the random chain 
$\{W(k)\}$ describes the evolution of the vector $x(k)$ that has its 
$i$-th component value equal to agent $i$ value, $x_i(k)$. 
As seen from~\eqref{eqn:gossip}, any realization of the model 
$\{W(k)\}$ is doubly stochastic. Hence, the model has a common steady state 
$\pi=\frac{1}{m}e$ in expectation. Also, the model has strong feedback property
(with $\gamma=\frac{1}{2}$). 

\begin{lemma}\label{lemma:gossip2}
For extended gossip algorithm~\eqref{eqn:gossip}, the consensus is almost 
sure if $\sum_{k=0}^\infty P_S(k)=\infty$ for any nontrivial set $S\subset[m]$.
\end{lemma}
\begin{proof}
The extended gossip algorithm satisfies the assumption of 
the infinite flow Theorem~\ref{thm:infflowthm}. 
Since $\EXP{W_{ij}(k)+W_{ji}(k)}=P_{ij}(k)$ for all $i\ne j$ and $k\ge0$, 
it follows that  
$\sum_{k=0}^\infty P_S(k)
=\sum_{k=0}^{\infty} \EXP{W_S(k)}$.
Thus, by the infinite flow theorem the model admits consensus if $\sum_{k=0}^\infty P_S(k)=\infty$ for any $S\subset[m]$.
\end{proof}

When $P(k)=P$ for all $k$ as in~\cite{Boyd05,Boyd06}, we can 
consider the graph $G=([m],E)$ where 
the edge $\{i,j\}\in E$ if and only if $P_{ij}=P_{ji}>0$. In this case,
it can be seen that the condition of Lemma~\ref{lemma:gossip2} 
is equivalent to the requirement that the graph $G$ is connected. 
One can further modify the algorithm in~\eqref{eqn:gossip}
to allow for time-varying weights, i.e.,  
$x_i(k)=a(k)x_i(k-1)+(1-a(k))x_j(k-1)$ and 
$x_j(k)=a(k)x_j(k-1)+(1-a(k))x_i(k-1)$ with 
$a(k)\in(0,1)$. Such a scheme is a natural 
generalization of the symmetric gossip model proposed in~\cite{Zampieri}. 
In this case, it can be verified that if 
$a(k)\in [a,1-a]$ for $a\in(0,\frac{1}{2}]$, 
then the result of Lemma~\ref{lemma:gossip2} still holds.

\section{Conclusion}\label{sec:conclusion}
We have studied the ergodicity and consensus problem 
for a linear discrete-time dynamic model driven by random stochastic matrices.
We have introduced a concept of the infinite flow event and 
studied the relations among this event, ergodicity event and consensus event.
The central result is the infinite flow theorem 
providing necessary and sufficient conditions for ergodicity of independent
random models. The theorem
captures the conditions ensuring the convergence of the 
random consensus algorithms, such as gossip and broadcast schemes
\cite{Aysal09,Dimakis08,Boyd06}.
Moreover, the infinite flow theorem captures simultaneously
the conditions on the connectivity of the system and the sufficient 
information flow over time that have been important in studying the consensus 
and average consensus in deterministic settings \cite{Tsitsiklis84,
TsitsiklisAthans84,Jadbabaie03,Olshevsky09,Olshevsky09a,Nedic_cdc07,Nedic09}.
As illustrated briefly on two examples,
the infinite flow theorem provides a convenient tool for 
studying the ergodicity of a model as well as consensus algorithms.
Finally, we note that the work in this paper is readily extendible to 
the case when the initial state $x(0)$  in~\eqref{eqn:dynsys} is 
itself random and independent of the chain $\{W(k)\}$.

\section*{Acknowledgement}
The authors are grateful to the anonimus referees for their valuable comments
and suggestions that improved the paper.

\bibliographystyle{IEEEtran}
\bibliography{consensus}
\end{document}